\documentclass[12pt]{amsart}
\usepackage[margin=1in]{geometry}

\usepackage{amssymb}
\usepackage{mathtools}
\usepackage{enumitem}
\usepackage{microtype}

\usepackage{amsrefs}
\usepackage{hyperref}
\makeatletter
\def\namedlabel#1#2{\begingroup
	#2%
	\def\@currentlabel{#2}%
	\phantomsection\label{#1}\endgroup}
\makeatother

\usepackage{tikz}
\usetikzlibrary{arrows}
\usetikzlibrary{calc}
\usetikzlibrary{decorations.pathreplacing}

\newcommand\D{\mathbb D}
\newcommand\R{\mathbb R}
\newcommand\T{\mathbb T}
\newcommand\BB{\mathcal B}
\newcommand\DD{\mathcal D}
\newcommand\LL{\mathcal L}
\newcommand\PP[1]{\textnormal{\ref{P#1}}}

\DeclareMathOperator*{\esssup}{ess\,sup}
\DeclareMathOperator*{\essinf}{ess\,inf}

\theoremstyle{plain}
\newtheorem{theorem}{Theorem}[section]
\newtheorem{proposition}[theorem]{Proposition}
\newtheorem{lemma}[theorem]{Lemma}
\newtheorem{corollary}[theorem]{Corollary}

\theoremstyle{definition}
\newtheorem{definition}[theorem]{Definition}
\newtheorem{example}[theorem]{Example}

\theoremstyle{remark}
\newtheorem{remark}[theorem]{Remark}

\AtBeginEnvironment{example}{
	\pushQED{\qed}}
\AtEndEnvironment{example}{
	\popQED\endexample}

\numberwithin{equation}{section}

\begin{document}

\title[Sharpness of a condition for B\'ekoll\'e-Bonami weights]{Sharpness of the side condition in a characterization of B\'ekoll\'e-Bonami weights}

\author[A. C. Goksan]{Alptekin Can Goksan}
\address{Department of Mathematics, University of Toronto, Toronto, Ontario, Canada}
\email{a.goksan@mail.utoronto.ca}


\begin{abstract}

We study the sharpness of the side condition in a recent characterization of a limiting class $B_\infty$ of B\'ekoll\'e-Bonami weights by Aleman, Pott and Reguera. This side condition bounds the oscillation of a weight on the top halves of Carleson squares and allows for the development of a rich theory for B\'ekoll\'e-Bonami weights, analogous to that of Muckenhoupt weights. First, we prove that the side condition can essentially be dropped when the weight is radial and monotonic. Then, by means of counterexamples, we show that the side condition is sharp for non-monotonic weights. In addition, we extend the characterization of the $B_\infty$ class so that it includes all twelve $A_\infty$ conditions recently studied by Duoandikoetxea, Mart\'in-Reyes and Ombrosi, and we present a complete picture of the relationships between these twelve conditions for arbitrary weights on the unit disc. Finally, we use our results to prove an analogue of the self-improvement property of Muckenhoupt weights for monotonic B\'ekoll\'e-Bonami weights.

\end{abstract}

\maketitle

\section{Introduction}
\label{sec_intro}

Weighted norm inequalities constitute one of the main branches of harmonic analysis and are closely connected to singular integrals, which play an important role in partial differential equations, operator theory, and other areas. In the early 1970s, Hunt, Muckenhoupt and Wheeden \cite{HMW} characterized the weights $w$ such that the Hilbert transform is bounded on $L^p(w)$, and this led to the definition of the $A_p$ class of weights. Their work was extended by Coifman and Fefferman \cite{CF} to all Calder\'on-Zygmund operators, and a tremendous amount of work has been done on $A_p$ weights since then (see e.g.\ \cite{D}, \cite{G}, and the references therein). In the late 1970s, B\'ekoll\'e and Bonami \cite{BB} characterized the weights $w$ on the unit disc such that the Bergman projection is bounded on $L^p(w)$. These weights are called $B_p$ weights, and their definition closely resembles that of $A_p$ weights (see Definition \ref{def_disc}(b)).

The theory of $A_p$ weights is highly developed, mostly thanks to the reverse H\"older inequality, which was first proved in \cite{CF} and which paves the way for a number of other desirable properties, such as self-improvement to a lower $A_p$ class. Unfortunately, $B_p$ weights do not in general satisfy the reverse H\"older inequality (see Definition \ref{def_disc}(d)), and this has hindered the development of an analogous theory for $B_p$ weights.

Recently, certain subclasses of $B_p$ weights have attracted interest as they allow parts of the $A_p$ theory to be recovered. For example, Borichev \cite{B} has shown that $B_p$ weights of the form $\exp u$, where $u$ is a subharmonic function, satisfy a self-improvement property similar to that of $A_p$ weights. More recently, Aleman, Pott and Reguera \cite{APR1} have identified a broad class of $B_p$ weights which satisfy most of the desirable properties of $A_p$ weights, thus opening the door to a theory as fruitful as the $A_p$ theory. We focus on this class, which consists of weights that are ``almost constant'' on the top halves of Carleson squares, i.e.\ weights $w$ for which there is a constant $C \ge 1$ such that
\begin{equation}
	\label{eq_AC}
	C^{-1} w(z_2) \le w(z_1) \le C w(z_2)
\end{equation}
for every interval $I$ of the unit circle and a.e.\ $z_1, z_2 \in T_I$, where $T_I$ is the top half of the Carleson square $Q_I$ associated with $I$. (See \eqref{eq_Q_I}, \eqref{eq_T_I}, and also Definition \ref{def_AC}.) The following is part of the main theorem in \cite{APR1}:

\begin{theorem}[Aleman, Pott and Reguera, 2017]
	\label{thm_equiv_1}
	Let $w$ be a weight on the unit disc which is almost constant on top halves. Then the following are equivalent:
	\begin{enumerate}[label=\upshape(\alph*)]
		\item $w \in B_p$ for some $1<p<\infty$.
		\item $w$ has the reverse H\"older property.
		\item $w$ has the reverse Jensen property.
		\item $w \in B_\infty$.
	\end{enumerate}
\end{theorem}

The two properties which have been omitted here will be addressed separately in Section \ref{sec_further}. The reverse Jensen and $B_\infty$ properties mentioned in Theorem \ref{thm_equiv_1} (see Definition \ref{def_disc}(c) and (e)) are the unit disc analogues of two different definitions of the $A_\infty$ property for weights on $\R^n$. The former definition was introduced independently by Garc\'ia-Cuerva and Rubio de Francia \cite{GR} and Hru\v{s}\v{c}ev \cite{H}, and the latter definition, which involves a maximal function, was introduced by Fujii \cite{F} and also studied by Wilson \cite{W}. It is well-known that the $\R^n$ counterparts of properties (a) - (d) in Theorem \ref{thm_equiv_1} are equivalent (see e.g.\ \cite{GR}*{Section IV.2} and \cite{F}). Let us also mention that the unit disc analogue of the $A_1$ property is called the $B_1$ property (see Definition \ref{def_disc}(a)). When we say ``unit disc analogue'', we mean that cubes or balls in the definition of the property are replaced with Carleson squares.

The notion of a weight being almost constant on top halves is relatively new in the literature. A version of it seems to have first appeared in \cite{B}, where weights that are constant on the top halves associated with a dyadic grid are considered. The formulation \eqref{eq_AC} is equivalent to a property called bounded hyperbolic oscillation \cite{HKZ}*{Lemma 2.17}, which has recently been studied by Limani and Nicolau \cite{LN} and by Dayan, Llinares and Perfekt \cite{DLP}. In the recent preprint \cite{MP}, Mudarra and Perfekt introduce new side conditions involving dyadic maximal and minimal operators and deduce a nice alternative characterization of properties (a) - (d) in Theorem \ref{thm_equiv_1}. Our viewpoint is different from \cite{MP}, since we focus on conditions on the growth, decay, or oscillation of the weight, as expressed by \eqref{eq_AC}.

In this paper, we address the following natural question: Is the side condition in Theorem \ref{thm_equiv_1} sharp? In other words, can we replace the assumption that $w$ is almost constant on top halves with a weaker assumption on the growth, decay, or oscillation of $w$, and keep the same nice equivalences? Our first result is that, for weights which are radial and monotonic (from the centre of the disc to the boundary), the hypothesis of Theorem \ref{thm_equiv_1} is satisfied, provided that at least one of the equivalent conditions is satisfied, except in one specific case. To be precise, we have the following:

\begin{theorem}
	\label{thm_mono_1}
	Let $w$ be a weight on the unit disc which is radial and satisfies $0 < w(0) < \infty$. Suppose either $w$ is increasing and satisfies one of conditions \textup{(a)} - \textup{(d)} in Theorem \textup{\ref{thm_equiv_1}}, or $w$ is decreasing and satisfies one of conditions \textup{(a)} - \textup{(c)} in Theorem \textup{\ref{thm_equiv_1}}. Then $w$ is almost constant on top halves.
\end{theorem}

The reader will notice that the case where $w$ is decreasing and satisfies condition (d) is missing from Theorem \ref{thm_mono_1}. It turns out that condition (d) holds automatically for decreasing weights; see Proposition \ref{prop_dec}. Moreover, the assumption in Theorem \ref{thm_mono_1} that $0<w(0)<\infty$ cannot be dropped; see Example \ref{ex_centre}.

Together, Theorems \ref{thm_equiv_1} and \ref{thm_mono_1} imply that the reverse H\"older property, the reverse Jensen property, and the property of belonging to some $B_p$ class with $1<p<\infty$ are equivalent for monotonic radial weights; see Corollary \ref{cor_2}. The two theorems also lead to a self-improvement property for such weights; see Corollary \ref{cor_3}.

Let us briefly describe the key ideas in the proof of Theorem \ref{thm_mono_1}. First, we carry out a change of variables to translate the two-dimensional problem into a one-dimensional problem. Then, we argue by contradiction. Notice that, if a monotonic weight $w$ fails to be almost constant on top halves, then $w$ exhibits arbitrarily large variation (in multiplicative terms) between the inner and outer edges of a top half. However, this does not immediately imply that $w$ is convex enough for one of conditions (a) - (d) to fail. To overcome this difficulty, we either (if $w$ is decreasing) subdivide the top half into small pieces and pick the one on which $w$ varies the most, or (if $w$ is increasing) choose the top half so that it occupies a small portion of the unit disc close to the boundary.

It is natural to ask what happens if the assumption of monotonicity in Theorem \ref{thm_mono_1} is dropped. In other words, if a weight $w$ satisfies one of conditions (a) - (d) in Theorem \ref{thm_equiv_1}, is $w$ necessarily almost constant on top halves? In Section \ref{sec_counter}, we present counterexamples which disprove this and other possible implications between the properties that we have discussed, for weights which are radial but not monotonic. These counterexamples are summarized in Figure \ref{fig_map_1}. In particular, Example \ref{ex_6} shows that Theorem \ref{thm_mono_1} fails completely for non-monotonic weights.

\begin{figure}
	\begin{tikzpicture}[scale=1.125]
		\def \cross {\Large $\times$};
		\tikzset{arrow/.style={-implies, double equal sign distance}};
		\tikzset{txt/.style={font=\Small, align=center}};
		
		\draw (2,8) node {$B_p$};
		\draw (6,8) node {$B_1$};
		\draw (2,5) node {$RJ$};
		\draw (6,5) node {$AC$};
		\draw (2,2) node {$B_\infty$};
		\draw (6,2) node {$RH$};
		
		\draw[arrow] (5.5,8) to (2.5,8);
		\draw[arrow] (2,7.5) to (2,5.5);
		\draw[arrow] (2,4.5) to (2,2.5);
		\draw[arrow] (5.5,2) to (2.5,2);
		
		\draw[arrow, densely dashed] (1.75,5.5) to [bend left=21pt]
		node[txt, left=3pt] {if $w \in AC$} (1.75,7.5);
		\draw[arrow, densely dashed] (1.75,2.5) to [bend left=21pt]
		node[txt, left=3pt] {if $w \in AC$} (1.75,4.5);
		\draw[arrow, densely dashed] (2.5,1.75) to [bend right=15pt]
		node[txt, below=3pt] {if $w \in AC$} (5.5,1.75);
		
		\draw[arrow, densely dashed, thick] (2.5,5) to
		node[txt, above=1pt] {if $w \in DEC$} (5.5,5);
		\draw[arrow, densely dashed, thick] (2.4,2.5) to
		node[txt, pos=2/3, rotate=atan(2.2/2.9), above=1pt] {if $w \in INC$} (5.3,4.7);
		\draw[arrow, densely dashed, thick] (6,2.5) to
		node[txt, rotate=90, above=1pt] {if $w \in DEC$} (6,4.5);
		\draw[thick] (2,2) circle[radius=10pt]
		node[txt, below left=5pt] {automatic \\ if $w \in DEC$};
		
		\draw[arrow, thick] (6,5.5) to node {\cross} (6,7.5);
		\draw[densely dotted, thick] (6,6.5) to (6.75,6.8);
		\draw (7.5,6.8) node[txt, draw, densely dotted, thick] {even if \\ $w \in B_\infty$ \\ {[Ex \ref{ex_9}]}};
		\draw[arrow, thick] (5.4,4.45) to node[pos=1/3, rotate=atan(2.2/2.9)] {\cross}
		node[txt, pos=1/3, right=12pt, below=4.5pt] {[Ex \ref{ex_10}]} (2.5,2.25);
		
		\draw[arrow, thick] (6.25,7.5) to [bend left=15pt] node[pos=2/3, rotate=-7.5] {\cross}
		node[txt, pos=2/3, right=3pt] {[Ex \ref{ex_1}]} (6.25, 2.5);
		\draw[arrow, thick] (5.5,2.25) to [bend left=21pt] node[pos=4/5, rotate=-48] {\cross}
		node[txt, pos=4/5, right=12pt, above=6pt] {[Ex \ref{ex_3}]} (2.25,4.5);
		\draw[arrow, thick] (2.25,5.5) to [bend right=21pt] node {\cross}
		node[txt, right=3pt] {even if \\ $w \in RH$ \\ {[Ex \ref{ex_4}]}} (2.25,7.5);
		\draw[arrow, thick] (5.75,7.5) to [bend right=21pt] node {\cross}
		node[txt, left=3pt] {even if \\ $w \in RH$ \\ {[Ex \ref{ex_6}]}} (5.75,5.5);
		\draw[arrow, thick] (2.5,8.25) to [bend left=15pt] node {\cross}
		node[txt, above=3pt] {even if $w \in RH$ [Ex \ref{ex_7}]} (5.5,8.25);
		
		\matrix[draw, font=\Small, nodes={align=left}, cells={anchor=north west}] at (10.7,5)
		{
			\node {$RJ$}; & \node {reverse Jensen}; \\
			\node {$RH$}; & \node {reverse H\"older}; \\
			\node {$AC$}; & \node {almost constant \\ on top halves}; \\
			\node {$INC$}; & \node {increasing and \\ $0<w(0)<\infty$}; \\
			\node {$DEC$}; & \node {decreasing and \\ $0<w(0)<\infty$}; \\
			\node {}; \\
			\node[right=2pt, below=4pt] (A) {}; \node[right of=A, node distance=30pt] (B) {};
			\draw[arrow] (A) to (B); & \node {true implication};\\
			\node[right=2pt, below=4pt] (A) {}; \node[right of=A, node distance=30pt] (B) {};
			\draw[arrow, dashed] (A) to (B); & \node {conditional \\ implication};\\
			\node[right=2pt, below=4pt] (A) {}; \node[right of=A, node distance=30pt] (B) {};
			\draw[arrow] (A) to (B); \node[right of=A, node distance=13pt] {\cross}; & \node {false implication};\\
		};
	\end{tikzpicture}
	\caption{Relationships between the properties in Definitions \ref{def_disc} and \ref{def_AC}, for arbitrary weights on the unit disc. The label $B_p$ represents $\bigcap_{1<p<\infty} B_p$ in the two implications to its right and $\bigcup_{1<p<\infty} B_p$ in the three implications below it. Thick arrows indicate contributions of this paper.}
	\label{fig_map_1}
\end{figure}
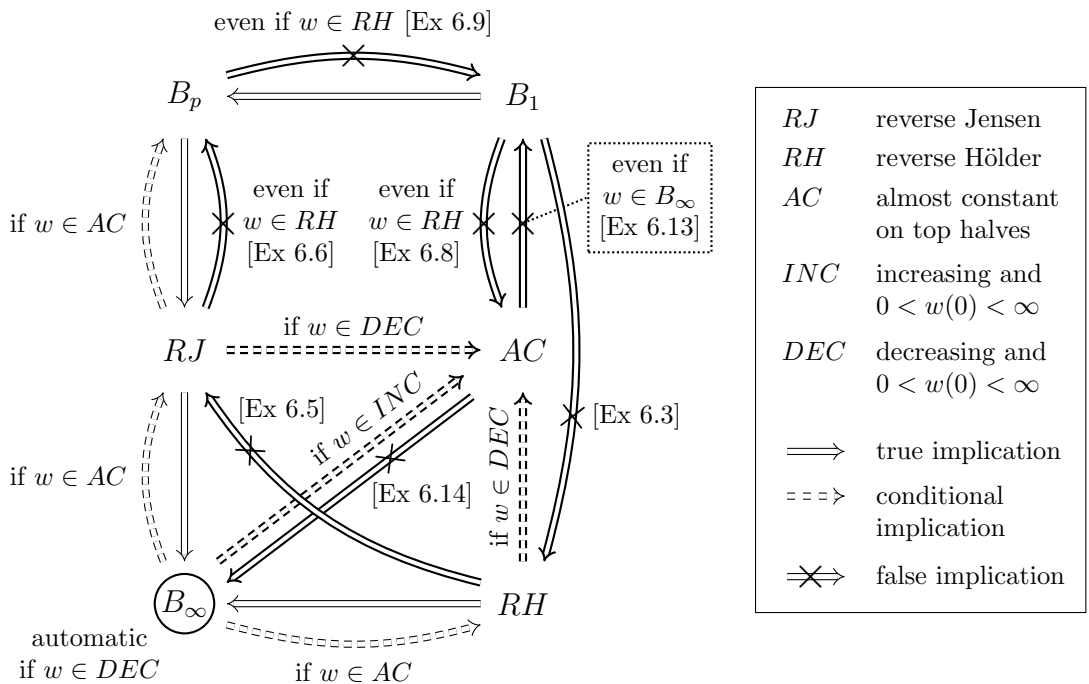

The main construction behind most of our counterexamples is a radial weight which is piecewise constant and oscillates between a fixed value and increasingly larger or smaller values as the boundary of the disc is approached. By judicious choice of the sizes of the exceptional sets and the values of the weight on the exceptional sets, we obtain a wide variety of counterexamples.

Our counterexamples show that the side condition in Theorem \ref{thm_equiv_1} is sharp. Indeed, in Examples \ref{ex_1} - \ref{ex_7} and \ref{ex_8}, the oscillations of the weight (in multiplicative terms) can be chosen to go to infinity as slowly as desired. By definition, a weight is almost constant on top halves if and only if its oscillations on top halves (in multiplicative terms) are bounded. Examples \ref{ex_1}, \ref{ex_3} and \ref{ex_4} show that even the slightest relaxation of this condition causes the equivalences in Theorem \ref{thm_equiv_1} to fail. Therefore, the hypothesis of Theorem \ref{thm_equiv_1} cannot be improved. (See also Remark \ref{rmk_osc}.)

Other conclusions that can be drawn from our counterexamples include the following: Example \ref{ex_9} shows that, even for monotonic weights, the $B_1$ condition cannot be included in the list of equivalent conditions in Theorem \ref{thm_equiv_1}. Example \ref{ex_10} shows that monotonic weights that are almost constant on top halves need not satisfy any of these four conditions.

Together, Theorems \ref{thm_equiv_1} and \ref{thm_mono_1}, Propositions \ref{prop_dec} and \ref{prop_arb_2}, and Examples \ref{ex_1}, \ref{ex_3}, \ref{ex_4}, \ref{ex_6}, \ref{ex_7}, \ref{ex_9} and \ref{ex_10} allow us to construct a complete picture of the relationships between the properties that we have mentioned so far; see Figure \ref{fig_map_1}. In particular, we have the following:

\begin{theorem}
	\label{thm_map_1}
	For arbitrary weights on the unit disc, the truth or falsity of any pairwise implication between the following six properties is given by Figure \textup{\ref{fig_map_1}}: $B_1$, $B_p~(1<p<\infty)$, $B_\infty$, reverse H\"older, reverse Jensen, and almost constant on top halves.
\end{theorem}

In fact, even more can be said. Properties (a) - (d) in Theorem \ref{thm_equiv_1} are part of a larger collection of properties which are equivalent for weights on $\R^n$ and are collectively referred to as the $A_\infty$ property. Recently, Duoandikoetxea, Mart\'in-Reyes and Ombrosi \cite{DMO1} carried out a comprehensive study of these properties in the context of a general measure space equipped with a ``basis'', such as the basis of cubes (or balls) in $\R^n$ and the basis of Carleson squares in the unit disc. They came up with an almost complete map of the implications between twelve of these properties, thus greatly clarifying how these properties relate to one another in a very general setting. Their map was later completed by Kosz \cite{K}. We list these twelve properties in Definition \ref{def_P1-P8} as \ref{P1} - \ref{P8} and \ref{P1'} - \ref{P4'}, in keeping with the notation in \cite{DMO1}. It is shown in \cite{DMO1} that \textbf{(P\textit{j}')} is always equivalent to \textbf{(P\textit{j})} for $j=1,\dots,4$, so we mainly focus on \ref{P1} - \ref{P8} in this paper.

The characterization in Theorem \ref{thm_equiv_1} can be extended so that it encompasses all of properties \ref{P1} - \ref{P8}, as follows:

\begin{theorem}
	\label{thm_equiv_2}
	For weights on the unit disc which are almost constant on top halves, properties \textup{\ref{P1}} - \textup{\ref{P8}} are equivalent.
\end{theorem}

By combining Proposition \ref{prop_arb_1} and Examples \ref{ex_1} - \ref{ex_6} and \ref{ex_8}, we arrive at the following complete picture of the relationships between properties \ref{P1} - \ref{P8}:

\begin{theorem}
	\label{thm_map_2}
	For arbitrary weights on the unit disc, the truth or falsity of any pairwise implication between properties \textup{\ref{P1}} - \textup{\ref{P8}} is given by Figure \textup{\ref{fig_map_2}}. In particular, no two of these properties are equivalent for such weights.
\end{theorem}

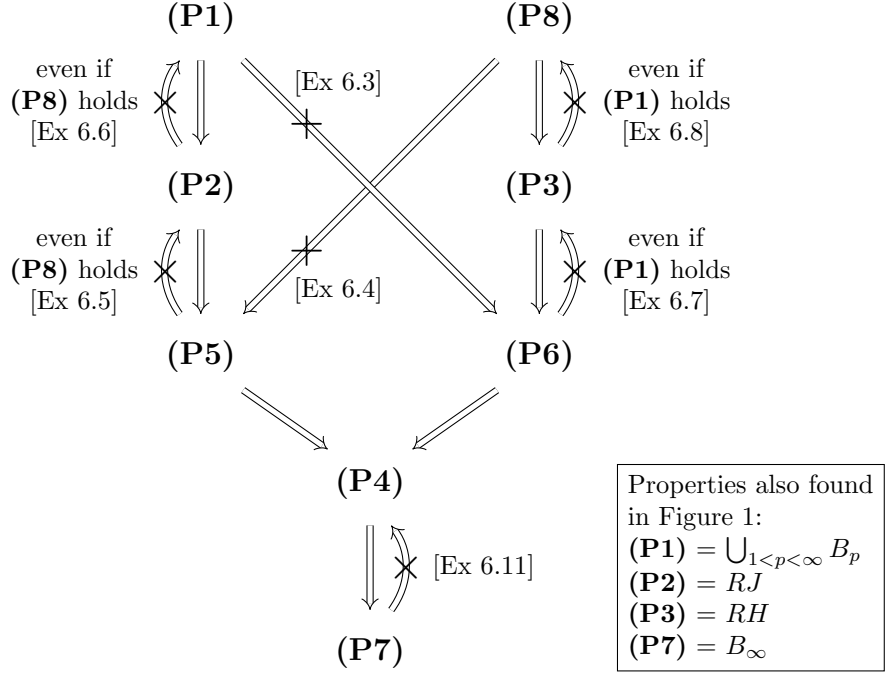
\begin{figure}
	\begin{tikzpicture}[scale=1.125]
		\def \cross {\Large $\times$};
		\tikzset{arrow/.style={-implies, double equal sign distance}};
		\tikzset{txt/.style={font=\Small, align=center}};
		
		\draw (2,8) node {\ref{P1}};
		\draw (2,6) node {\ref{P2}};
		\draw (2,4) node {\ref{P5}};
		\draw (6,8) node {\ref{P8}};
		\draw (6,6) node {\ref{P3}};
		\draw (6,4) node {\ref{P6}};
		\draw (4,2.5) node {\ref{P4}};
		\draw (4,0.5) node {\ref{P7}};
		
		\draw[arrow] (2,7.5) to (2,6.5);
		\draw[arrow] (2,5.5) to (2,4.5);
		\draw[arrow] (6,7.5) to (6,6.5);
		\draw[arrow] (6,5.5) to (6,4.5);
		\draw[arrow] (2.5,3.6) to (3.5,2.9);
		\draw[arrow] (5.5,3.6) to (4.5,2.9);
		\draw[arrow] (4,2) to (4,1);
		
		\draw[arrow] (1.75,6.5) to [bend left=36pt] node {\cross}
		node[txt, left=6pt] {even if \\ \ref{P8} holds \\ {[Ex \ref{ex_4}]}} (1.75,7.5);
		\draw[arrow] (1.75,4.5) to [bend left=36pt] node {\cross}
		node[txt, left=6pt] {even if \\ \ref{P8} holds \\ {[Ex \ref{ex_3}]}} (1.75,5.5);
		\draw[arrow] (6.25,6.5) to [bend right=36pt] node {\cross}
		node[txt, right=6pt] {even if \\ \ref{P1} holds \\ {[Ex \ref{ex_6}]}} (6.25,7.5);
		\draw[arrow] (6.25,4.5) to [bend right=36pt] node {\cross}
		node[txt, right=6pt] {even if \\ \ref{P1} holds \\ {[Ex \ref{ex_5}]}} (6.25,5.5);
		\draw[arrow] (4.25,1) to [bend right=36pt] node {\cross}
		node[txt, right=6pt] {[Ex \ref{ex_8}]} (4.25,2);
		\draw[arrow] (5.5,7.5) to node[pos=3/4, rotate=45] {\cross}
		node[txt, pos=3/4, right=12pt, below=6pt] {[Ex \ref{ex_2}]} (2.5,4.5);
		\draw[arrow] (2.5,7.5) to node[pos=1/4, rotate=45] {\cross}
		node[txt, pos=1/4, right=12pt, above=6pt] {[Ex \ref{ex_1}]} (5.5,4.5);
		
		\draw (8.5,1.5) node[draw, font=\Small, align=left] {Properties also found \\ in Figure \ref{fig_map_1}: \\ \ref{P1} = $\bigcup_{1<p<\infty} B_p$ \\ \ref{P2} = $RJ$ \\ \ref{P3} = $RH$ \\ \ref{P7} = $B_\infty$};
	\end{tikzpicture}
	\caption{Relationships between properties \ref{P1} - \ref{P8} in Definition \ref{def_P1-P8}, for arbitrary weights on the unit disc.}
	\label{fig_map_2}
\end{figure}

Moreover, we deduce Theorem \ref{thm_mono_1} as a consequence of the following more general result:

\begin{theorem}
	\label{thm_mono_2}
	Let $w$ be a weight on the unit disc which is radial and satisfies $0<w(0)<\infty$. Suppose either $w$ is increasing and satisfies \textup{\ref{P7}}, or $w$ is decreasing and satisfies \textup{\ref{P4}}. Then $w$ is almost constant on top halves.
\end{theorem}

In Section \ref{sec_further}, we add two more conditions to the list of $A_\infty$ conditions (see Definition \ref{def_P4ab}), and we indicate how these two conditions are related to \ref{P1} - \ref{P8} (see Figure \ref{fig_map_P4}).

This paper is organized as follows: Section \ref{sec_def} contains the definitions, Section \ref{sec_prelim} contains a few preliminary results, Section \ref{sec_equiv} contains the proof of Theorem \ref{thm_equiv_2}, Section \ref{sec_mono} contains the proof of Theorem \ref{thm_mono_1} and a few corollaries, Section \ref{sec_counter} contains the counterexamples mentioned above, and Section \ref{sec_further} contains a discussion of three additional $A_\infty$ conditions.

\section{Definitions}
\label{sec_def}

We equip the unit circle $\T=\{|z|=1\}$ with normalized arc length measure and the unit disc $\D=\{|z|<1\}$ with normalized area measure. Whenever we write $I \subset \T$, we mean~that~$I$ is an interval (arc) of $\T$. For any $I \subset \T$, the \emph{Carleson square} associated with $I$ is defined by
\begin{equation}
	\label{eq_Q_I}
	Q_I = \left\{ z \in \D \setminus \{0\} : z/|z| \in I ~\text{and}~ 1 - |z| < |I| \right\},
\end{equation}
and the \emph{top half} of $Q_I$ is defined by
\begin{equation}
	\label{eq_T_I}
	T_I = \left\{z \in \D \setminus \{0\} : z/|z| \in I ~\text{and}~ |I|/2 < 1-|z| < |I| \right\}.
\end{equation}
(See Figure \ref{fig_Carleson}.) It is easy to check that $|Q_I| = |I|^2 (2-|I|)$ and $|T_I| = |I|^2 (1-\frac{3}{4}|I|)$. In particular, $|Q_I| \le 4|T_I|$. For any measurable function $f : \D \to \R$, the associated \emph{maximal function} is the function on $\D$ defined by
\[
	Mf(z) = \sup_{I \subset \T\,:\,z \in Q_I} \frac{1}{|Q_I|} \int_{Q_I} |f|.
\]

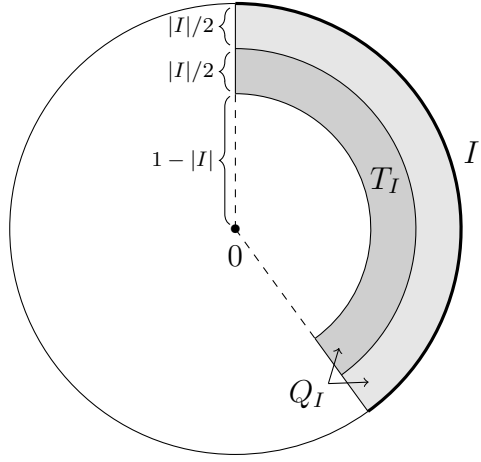
\begin{figure}
	\begin{tikzpicture}[x=3cm,y=3cm]
		\def \mod {3/5};
		\def \arg {18};
		\pgfmathsetmacro{\avg}{(1+\mod)/2};
		\pgfmathsetmacro{\angle}{180*(1-\mod)};
		\pgfmathsetmacro{\min}{\arg - \angle};
		\pgfmathsetmacro{\max}{\arg + \angle};
		\tikzset{brace/.style={decorate, decoration={brace, amplitude=4pt, raise=2pt}}};
		\tikzset{txt/.style={midway, left=4pt, font=\tiny}};
		
		\draw (\max:1) arc (\max:{\min + 360}:1);
		\draw[very thick] (\min:1) arc (\min:\max:1);
		\draw (\min:\mod) arc (\min:\max:\mod);
		\draw (\min:\avg) arc (\min:\max:\avg);
		
		\draw[dashed] (0,0) -- ($\mod*({cos(\max)},{sin(\max)})$);
		\draw[dashed] (0,0) -- ($\mod*({cos(\min)},{sin(\min)})$);
		\draw ($\mod*({cos(\max)},{sin(\max)})$) -- ($({cos(\max)},{sin(\max)})$);
		\draw ($\mod*({cos(\min)},{sin(\min)})$) -- ($({cos(\min)},{sin(\min)})$);
		
		\fill[opacity=0.1] (\min:\mod) arc (\min:\max:\mod) -- (\max:\avg) arc (\max:\min:\avg) -- cycle;
		\fill[opacity=0.1] (\min:\mod) arc (\min:\max:\mod) -- (\max:1) arc (\max:\min:1) -- cycle;
		
		\draw (0,0) node{\tiny $\bullet$} node[below=2pt]{0};
		\draw ($1.1*({cos(\arg)},{sin(\arg)})$) node{$I$};
		\draw (${(\mod+\avg)/2}*({cos(\arg)},{sin(\arg)})$) node{$T_I$};
		\draw ($\avg*({cos(\arg-\angle-12.5)},{sin(\arg-\angle-12.5)})$) node{$Q_I$};
		
		\draw[->] ($\avg*({cos(\min-5)},{sin(\min-5)})$) -- ($ {(\mod+\avg)/2}*({cos(\min+5)},{sin(\min+5)})$);
		\draw[->] ($\avg*({cos(\min-5)},{sin(\min-5)})$) -- ($ {(\avg+1)/2}*({cos(\min+5)},{sin(\min+5)})$);
		\draw[brace] ($0.02*({cos(\max)},{sin(\max)})$) -- (${0.98*\mod}*({cos(\max)},{sin(\max)})$) node[txt]{$1-|I|$};
		\draw[brace] (${1.02*\mod}*({cos(\max)},{sin(\max)})$) -- (${0.98*\avg}*({cos(\max)},{sin(\max)})$) node[txt]{$|I|/2$};
		\draw[brace] (${1.02*\avg}*({cos(\max)},{sin(\max)})$) -- ($0.98*({cos(\max)},{sin(\max)})$) node[txt]{$|I|/2$};
	\end{tikzpicture}
	\caption{An interval $I$ of the unit circle, the associated Carleson square $Q_I$, and its top half $T_I$.}
	\label{fig_Carleson}
\end{figure}

A measurable function $w : \D \to [0, \infty]$ such that $0 < w < \infty$ a.e.\ is called a \emph{weight}. For any measurable set $A\subset\D$, we define $w(A) = \int_A w$. If $|A|>0$, we define $w_A = w(A)/|A|$, the average of $w$ over $A$.

In what follows, $p'$ denotes the H\"older conjugate of an exponent $p$, $1_A$ denotes the indicator function of a set $A$, and $\log^+(x)\coloneq\max(\log(x),0)$.

\begin{definition}
	\label{def_disc}
	For any integrable weight $w$ on $\D$, we say that
	\begin{enumerate}[label=(\alph*)]
		\item $w \in B_1$ if there exists $C > 0$ such that
		\begin{equation*}
			w_{Q_I} \le C \essinf_{Q_I} w
		\end{equation*}
		for every $I\subset\T$ or, equivalently,
		\begin{equation*}
			Mw(z) \le Cw(z)
		\end{equation*}
		for a.e.\ $z \in \D$;
		\item $w \in B_p\;(1 < p < \infty)$ if there exists $C > 0$ such that, for every $I \subset \T$,
		\begin{equation*}
			\left( \frac{1}{|Q_I|} \int_{Q_I} w \right) \left( \frac{1}{|Q_I|} \int_{Q_I} w^{1-p'} \right)^{p-1} \le C;
		\end{equation*}
		\item $w$ has the \emph{reverse Jensen property} (or $w \in RJ$ for short) if there exists $C > 0$ such that, for every $I \subset \T$,
		\begin{equation*}
			\frac{1}{|Q_I|} \int_{Q_I} w \le C \exp{\left( \frac{1}{|Q_I|} \int_{Q_I} \log w \right)};
		\end{equation*}
		\item $w$ has the \emph{reverse H\"older property} (or $w \in RH$ for short) if there exist $1< q < \infty$ and $C > 0$ such that, for every $I \subset \T$,
		\begin{equation*}
			\left( \frac{1}{|Q_I|} \int_{Q_I} w^q \right)^{1/q} \le \frac{C}{|Q_I|} \int_{Q_I} w;
		\end{equation*}
		\item $w \in B_\infty$ if there exists $C > 0$ such that, for every $I \subset \T$,
		\begin{equation*}
			\int_{Q_I} M(w1_{Q_I}) \le C \int_{Q_I} w.
		\end{equation*}
	\end{enumerate}
\end{definition}

The class $B_\infty$ was introduced by Aleman, Pott and Reguera in \cite{APR2} and further studied by the same authors in \cite{APR1}, where the following property was also introduced:

\begin{definition}
	\label{def_AC}
	A weight $w$ on $\D$ is \emph{almost constant on top halves} (or $w \in AC$ for short) if there exists $C \ge 1$ such that, for every $I \subset \T$ and a.e.\ $z \in T_I$,
	\[
		C^{-1} w_{T_I} \le w(z) \le C w_{T_I},
	\]
	or, equivalently, if there exists $C \ge 1$ such that, for every $I \subset \T$,
	\[
		\esssup_{T_I} w \le C \essinf_{T_I} w.
	\]
	More generally, if $\mathcal{I}$ is any collection of intervals of $\T$, we say that $w$ is \emph{almost constant on top halves $T_I$ with $I \in \mathcal{I}$} if the same statement holds for every $I \in \mathcal{I}$ but perhaps not for every $I\subset\T$.
\end{definition}

In this paper, we say that a weight $w$ is \emph{radial} if there is a measurable function $f : [0, 1) \to [0, \infty]$ such that $w(z) = f(|z|)$ for all $z \in \D$. In this case, we say that $w$ is \emph{monotonic} (resp.\ \emph{increasing}, \emph{decreasing}) if $f$ is monotonic (resp.\ increasing, decreasing). We use the terms ``increasing'' and ``decreasing'' in the weak (non-strict) sense.

The following definitions due to Duoandikoetxea, Mart\'in-Reyes and Ombrosi \cite{DMO1} will help us carry out a more thorough investigation of $A_\infty$ conditions for weights on $\D$: Let $(X,\Sigma,\mu)$ be a $\sigma$-finite measure space. A \emph{basis} for $X$ is a collection $\BB$ of measurable subsets of $X$ such that $0<\mu(B)<\infty$ for all $B\in\BB$ and $\bigcup_{B\in\BB} B = X$ up to a set of measure zero. For any measurable function $f$ on $X$, we define the associated \emph{maximal function} by
\[
	Mf(x) = \sup_{B\in\BB\,:\,x \in B} \frac{1}{\mu(B)} \int_B |f|\,d\mu, \qquad x \in X.
\]
A measurable function $w$ on $X$ such that $0<w<\infty$ a.e.\ is called a \emph{weight} on $X$. For any measurable set $A \subset X$, we set $w(A) = \int_A w\,d\mu$. For any $B \in \BB$, we set $w_B = w(B)/\mu(B)$, and we define $m(w;B)$ to be the median of $w$ over $B$, i.e.\ a number $t\in(0,\infty)$ such that
\[
	\mu(\{x \in B : w(x)<t\}) \le \mu(B)/2
	\qquad \text{and} \qquad
	\mu(\{x \in B : w(x)>t\}) \le \mu(B)/2.
\]
The set of all such $t$ is a closed interval $[a,b]$, and $m(w;B)$ can be defined as $a$, $b$, or some combination thereof; the precise definition does not matter for our purposes.

\begin{definition}
	\label{def_P1-P8}
	As in \cite{DMO1}, for any weight $w$ on $X$ such that $w(B)<\infty$ for all $B\in\BB$, we define the following $A_\infty$ conditions:
	\begin{enumerate}
		\item[\namedlabel{P1}{\bfseries(P1)}] There exist $1<p<\infty$ and $C>0$ such that, for every $B\in\BB$,
		\[
			\left(\frac{1}{\mu(B)} \int_B w\,d\mu\right)
			\left(\frac{1}{\mu(B)} \int_B w^{1-p'}\,d\mu\right)^{p-1}
			\le C.
		\]
		\item[\namedlabel{P1'}{\bfseries(P1')}] There exist $\delta,C>0$ such that, for every $B \in \BB$ and every measurable set $E \subset B$,
		\[
			\frac{\mu(E)}{\mu(B)} \le C \left(\frac{w(E)}{w(B)}\right)^\delta.
		\]
		\item[\namedlabel{P2}{\bfseries(P2)}] There exists $C>0$ such that, for every $B\in\BB$,
		\[
			\frac{1}{\mu(B)} \int_B w\,d\mu
			\le C
			\exp\left(\frac{1}{\mu(B)} \int_B \log w\,d\mu\right).
		\]
		\item[\namedlabel{P2'}{\bfseries(P2')}] There exists $C>0$ such that, for every $B\in\BB$ and every $0<s<1$,
		\[
			\frac{1}{\mu(B)} \int_B w\,d\mu \le C \left(\frac{1}{\mu(B)} \int_B w^s\,d\mu\right)^{1/s}.
		\]
		\item[\namedlabel{P3}{\bfseries(P3)}] There exist $1<q<\infty$ and $C>0$ such that, for every $B\in\BB$,
		\[
			\left(\frac{1}{\mu(B)} \int_B w^q\,d\mu\right)^{1/q}
			\le \frac{C}{\mu(B)} \int_B w\,d\mu.
		\]
		\item[\namedlabel{P3'}{\bfseries(P3')}] There exist $\delta,C>0$ such that, for every $B\in\BB$ and every measurable set $E \subset B$,
		\[
			\frac{w(E)}{w(B)} \le C \left(\frac{\mu(E)}{\mu(B)}\right)^\delta.
		\]
		\item[\namedlabel{P4}{\bfseries(P4)}] There exist $\alpha,\beta\in(0,1)$ such that, for every $B\in\BB$ and every measurable set $E \subset B$,
		\[
			\mu(E)<\alpha\mu(B) \implies w(E)<\beta w(B).
		\]
		\item[\namedlabel{P4'}{\bfseries(P4')}] There exist $\alpha,\beta\in(0,1)$ such that, for every $B\in\BB$,
		\[
			\mu(\{x \in B : w(x) \le \alpha w_B\}) \le \beta \mu(B).
		\]
		\item[\namedlabel{P5}{\bfseries(P5)}] There exists $C>0$ such that, for every $B\in\BB$,
		\[
			w_B \le C m(w;B).
		\]
		\item[\namedlabel{P6}{\bfseries(P6)}] There exists $C>0$ such that, for every $B\in\BB$,
		\[
			\int_B w \log^+ \left(\frac{w}{w_B}\right) \,d\mu \le C w(B).
		\]
		\item[\namedlabel{P7}{\bfseries(P7)}] There exists $C>0$ such that, for every $B\in\BB$,
		\[
			\int_B M(w1_B)\,d\mu \le C w(B).
		\]
		\item[\namedlabel{P8}{\bfseries(P8)}] There exist $\beta,C>0$ such that, for every $B\in\BB$ and every $\lambda>w_B$,
		\[
			w(\{x \in B : w(x)>\lambda\})
			\le C \lambda \mu(\{x \in B : w(x)>\beta\lambda\}).
		\]
	\end{enumerate}
\end{definition}

The reader will notice that \ref{P1} is the union of the classes $A_p~(1<p<\infty)$, \ref{P2} is the reverse Jensen property, and \ref{P3} is the reverse H\"older property. Of the remaining properties, \ref{P5} was introduced by Str\"omberg and Torchinsky \cite{ST}, \ref{P6} and \ref{P7} were introduced by Fujii \cite{F}, and \ref{P4'} and \ref{P8} first appeared in the work of Coifman and Fefferman \cite{CF}. As we mentioned in Section \ref{sec_intro}, \textbf{(P\textit{j}')} is always equivalent to \textbf{(P\textit{j})} for $j=1,\dots,4$. It is also worth mentioning that a slightly different version of \ref{P4} (see \ref{P4a} in Section \ref{sec_further}) is the original definition of the $A_\infty$ condition due to Muckenhoupt \cite{M}, and that \ref{P3'} is the definition of $A_\infty$ adopted in \cite{CF}. The authors of \cite{DMO1} were motivated to study \ref{P1} - \ref{P8} in general measure spaces after observing in \cite{DMO2} that these properties are not equivalent for the ray $(0,\infty)$ equipped with the basis of intervals $(0,b)$, where $0<b<\infty$.

\begin{definition}
	\label{def_A1}
	A weight $w$ on $X$ is said to be an $A_1$ weight if there exists $C>0$ such that
	\[
		w_B \le C \essinf_{B} w
	\]
	for every $B\in\BB$ or, equivalently,
	\[
		Mw(x) \le Cw(x)
	\]
	for a.e.\ $x \in X$.
\end{definition}

In this paper, we consider $\D$ with the basis of Carleson squares, and $(0,1]$ with the basis of intervals $(0,x]$, where $0<x\le1$. Let us reiterate that, for weights on $\D$, the class $A_p$ is denoted by $B_p$ for $1 \le p < \infty$, and the class of weights satisfying \ref{P7} is denoted by $B_\infty$.

\section{Preliminaries}
\label{sec_prelim}

We begin with an estimate related to the $B_\infty$ condition.

\begin{lemma}
	\label{lem_sup}
	Let $w$ be a weight on $\D$. For $I \subset \T$ and $z \in Q_I$, define
	\[
		M_I w(z) = \sup_{J \subset I\,:\,z \in Q_J} w_{Q_J}.
	\]
	(Note that $M_Iw(z)$ is the same as $Mw(z)$ except that we require $J \subset I$.) Then
	\[
		\int_{Q_I} M(w1_{Q_I}) \le 4 \int_{Q_I} M_I w.
	\]
	Moreover, if $w$ is radial, then $M(w1_{Q_I})=M_Iw$ on $Q_I$.
\end{lemma}

\begin{proof}
	For each $z \in Q_I$, let $z^*$ be the point of $Q_I$ obtained by reflecting $z$ in the axis of symmetry of $Q_I$. By definition,
	\[
		M(w1_{Q_I})(z) = \sup_{J \subset \T\,:\,z \in Q_J} \frac{w(Q_I \cap Q_J)}{|Q_J|}.
	\]
	Suppose $J\subset\T$ is such that $z \in Q_J$. If $|J| \ge |I|$, rotating $J$ about the origin until $J \supset I$ increases $w(Q_I \cap Q_J)$ to $w(Q_I)$, and then shrinking $J$ until $J = I$ decreases $|Q_J|$ to $|Q_I|$. If $|J| \le |I|$, there are two cases, depending on whether $I \cap J$ is a single interval or a union of two intervals. (The latter can occur if $|I|>1/2$.) In the former case, if we rotate $J$ just enough so that $J \subset I$, then $w(Q_I \cap Q_J)$ will increase and $z$ will remain in $Q_J$. In the latter case, if we rotate $J$ in the direction which contributes more to $w(Q_I \cap Q_J)$, again just enough so that $J \subset I$, then $w(Q_I \cap Q_J)$ will lose at most half of its value, and the new $Q_J$ will contain either $z$ or $z^*$. This reasoning shows that
	\[
		M(w1_{Q_I})(z) \le 2M_I w(z) + 2M_I w(z^*).
	\]
	Integrating over $Q_I$ and using the fact that $\int_{Q_I} f(z^*)\,dz = \int_{Q_I} f(z)\,dz$ for any nonnegative measurable function $f$ on $Q_I$, we get the desired inequality.
	
	If $w$ is radial, then, in the case where $|J|\le|I|$ and $I \cap J$ is a union of two intervals, rotating $J$ in the direction of $z$, just enough so that $J \subset I$, increases $w(Q_I \cap Q_J)$ and keeps $z$ in $Q_J$, and we obtain $M(w1_{Q_I})(z)=M_Iw(z)$.
\end{proof}

Next, we establish a dictionary between radial weights on $\D$ and weights on $(0,1]$.

\begin{lemma}
	\label{lem_dict}
	Let $w$ be a weight on $\D$ and let $f$ be a weight on $(0,1]$. Suppose $w(z) = f(1-|z|^2)$ for all $z\in\D$. Suppose $I\subset\T$ and $x \in (0,1]$ satisfy $x = |I|(2-|I|)$, i.e.\ $|I|=1-\sqrt{1-x}$. Then the following hold:
	\begin{enumerate}[label=\upshape(\alph*)]
		\item For any measurable function $\varphi:(0,\infty)\to\R$, we have
		\[
			\frac{1}{|Q_I|} \int_{Q_I} (\varphi \circ w) = \frac{1}{x} \int_0^x (\varphi \circ f),
		\]
		provided that at least one of the integrals is defined. In particular, $w$ is integrable if and only if $f$ is integrable.
		\item For all $z\in\D$, we have
		\[
			Mw(z) = Mf(1-|z|^2).
		\]
		\item For all $z \in Q_I$ and all $t \in (0,x]$, we have
		\[
			M(w1_{Q_I})(z) = \sup_{J \subset I\,:\,z \in Q_J} w_{Q_J}
			\qquad \text{and} \qquad
			M(f1_{(0,x]})(t) = \sup_{t \le y \le x} f_{(0,y]}.
		\]
		\item We have
		\[
			\frac{1}{|Q_I|} \int_{Q_I} M(w1_{Q_I}) = \frac{1}{x} \int_0^x M(f1_{(0,x]}).
		\]
		\item We have
		\[
			\esssup_{Q_I} w = \esssup_{(0,x]} f
			\qquad \text{and} \qquad
			\essinf_{Q_I} w = \essinf_{(0,x]} f.
		\]
		\item If $E \subset (0,x]$ is measurable, then
		\[
			w(\{z \in Q_I : 1-|z|^2 \in E\}) = |I|f(E).
		\]
		In particular,
		\[
			|\{z \in Q_I : 1-|z|^2 \in E\}| = |I||E|.
		\]
		\item If $S \subset (0,\infty)$ is measurable, then
		\[
			w(\{z \in Q_I : w(z) \in S\}) = |I| f(\{t \in (0,x] : f(t) \in S\}).
		\]
		In particular,
		\[
			|\{z \in Q_I : w(z) \in S\}| = |I| |\{t \in (0,x] : f(t) \in S\}|.
		\]
		\item We have
		\[
			m(w;Q_I) = m(f;(0,x]).
		\]
	\end{enumerate}
\end{lemma}

\begin{proof}
	(a) Note that $I = \{z\in\T : \theta_1 < \arg z < \theta_2\}$ for some $\theta_1,\theta_2\in\R$ with $0<\theta_2-\theta_1\le2\pi$. Hence $Q_I = \{z\in\D : |z|>1-|I| ~\text{and}~ \theta_1<\arg z<\theta_2\}$ and $\theta_2-\theta_1 = 2\pi|I|$. Using polar coordinates and the substitution $u=1-r^2$, we get
	\begin{align*}
		\int_{Q_I} \varphi(w(z))\,dz
		&= \frac{1}{\pi} \int_{\theta_1}^{\theta_2} \int_{1-|I|}^1 \varphi(w(re^{i\theta})) r\,dr\,d\theta
		= 2|I| \int_{1-|I|}^1 \varphi(f(1-r^2))r\,dr \\
		&= |I| \int_0^{|I|(2-|I|)} \varphi(f(u))\,du.
	\end{align*}
	Since $|Q_I| = |I|^2 (2-|I|)$, the desired formula follows. Taking $I=\T$, $x=1$, and $\varphi(t)=t$, we get $\int_\D w = \int_0^1 f$, so $\int_\D w < \infty$ if and only if $\int_0^1 f < \infty$.
	
	(b) For any measurable function $f : (0,1] \to \R$, we have
	\[
		Mf(t) = \sup_{t \le y \le 1} \frac{1}{y} \int_0^y |f|,
		\qquad t \in (0,1].
	\]
	If $J\subset\T$ and $z \in Q_J$, then $|z|>1-|J|$ and hence $1-|z|^2<|J|(2-|J|)\le1$. Conversely, if $z\in\D$ and $1-|z|^2<y\le1$, let $\alpha=1-\sqrt{1-y}$; then $y=\alpha(2-\alpha)$ and hence $|z|>1-\alpha$, so $z \in Q_J$ for some $J\subset\T$ with $|J|=\alpha$. Now, for any $z \in \D$, by part (a) and what we have just proved,
	\begin{align*}
		Mw(z)
		&= \sup_{J \subset \T\,:\,z \in Q_J} \frac{1}{|Q_J|} \int_{Q_J} w
		= \sup_{J \subset \T\,:\,z \in Q_J} \frac{1}{|J|(2-|J|)} \int_0^{|J|(2-|J|)} f \\
		&= \sup_{1-|z|^2 < y \le 1} \frac{1}{y} \int_0^y f
		= Mf(1-|z|^2).
	\end{align*}

	(c) For the first part, see Lemma \ref{lem_sup}. For the second part, let $t \in (0,x]$. By definition,
	\[
		M(f1_{(0,x]})(t) = \sup_{t \le y \le 1} \frac{1}{y} \int_0^{\min(x,y)} f.
	\]
	If $x \le y \le 1$, replacing $y$ with $x$ increases the value of $\frac{1}{y} \int_0^{\min(x,y)} f$. Thus, it suffices to take the supremum over $t \le y \le x$.
	
	(d) We compute
	\begin{align*}
		\frac{1}{|Q_I|} \int_{Q_I} M(w1_{Q_I})(z)\,dz
		&= \frac{1}{|Q_I|} \int_{Q_I} \left(\sup_{J \subset I\,:\,z \in Q_J} \frac{1}{|Q_J|} \int_{Q_J} w\right)\,dz \\
		&= \frac{1}{|Q_I|} \int_{Q_I} \biggl(\sup_{J \subset I\,:\,z \in Q_J} \frac{1}{|J|(2-|J|)} \int_0^{|J|(2-|J|)} f\biggr)\,dz \\
		&= \frac{1}{|Q_I|} \int_{Q_I} \biggl(\sup_{1-|z|^2 < y \le x} \frac{1}{y} \int_0^y f\biggr)\,dz \\
		&= \frac{1}{|Q_I|} \int_{Q_I} M(f1_{(0,x]})(1-|z|^2)\,dz \\
		&= \frac{1}{x} \int_0^x M(f1_{(0,x]})(t)\,dt.
	\end{align*}
	Here we used part (c), part (a), the fact that $1-|z|<|J|\le|I|$ if and only if $1-|z|^2<|J|(2-|J|)\le|I|(2-|I|)$, part (c) again, and finally a computation in polar coordinates as in part (a).
	
	(e) Since $|z|>1-|I|$ if and only if $1-|z|^2<x$, we have
	\[
		\esssup_{z \in Q_I} w(z)
		= \esssup_{1-|I|<|z|<1} f(1-|z|^2)
		= \esssup_{0<t<x} f(t).
	\]
	The proof for $\essinf$ is similar.
	
	(f) Let $\theta_1$ and $\theta_2$ be as in part (a). Let $A = \{r \in [0,1)\,:\,1-r^2 \in E\}$. Since $|z|>1-|I|$ if and only if $1-|z|^2<x$, we have
	\begin{align*}
		w(\{z \in Q_I : 1-|z|^2 \in E\})
		&= w(\{z \in \D : 1-|z|^2 \in E ~\text{and}~ \theta_1 < \arg z < \theta_2\}) \\
		&= \frac{1}{\pi} \int_{\theta_1}^{\theta_2} \int_A w(re^{i\theta})r\,dr\,d\theta
		= 2|I| \int_A f(1-r^2)r\,dr
		= |I|f(E).
	\end{align*}
	For the special case, take $w=1$ and $f=1$.
	
	(g) Set $E = \{t \in (0,x] : f(t) \in S\}$ in part (f).
	
	(h) Since $|Q_I|=|I||(0,x]|$, part (g) gives, for any $m\in\R$,
	\[
		|\{z \in Q_I : w(z) < m\}| \le |Q_I|/2
		\iff
		|\{t \in (0,x] : f(t) < m\}| \le |(0,x]|/2,
	\]
	and the same equivalence for $w(z)>m$ and $f(t)>m$.
\end{proof}

Now, we use our dictionary to show that, to determine whether a radial weight on $\D$ has a certain property, it often suffices to check whether the corresponding weight on $(0,1]$ has that property.

\begin{lemma}
	\label{lem_equiv}
	Let $w$ and $f$ be as in Lemma \textup{\ref{lem_dict}}, i.e.\ $w(z) = f(1-|z|^2)$. Then, for any $(\star) \in \{\PP1,\dots,\PP8,A_1\}$, we have $w \in (\star) \iff f \in (\star)$. Moreover, any constant(s) (one or two of $C$, $p$, $q$, $\alpha$, $\beta$) that work(s) for one of the two weights also work(s) for the other.
\end{lemma}

\begin{proof}
	For $(\star)\ne\PP4$, this follows immediately from Lemma \ref{lem_dict}. For $(\star)=\PP4$, a slightly longer argument is needed:
	
	$(\implies)$ Suppose $w \in \PP 4$. Let $x \in (0,1]$ and a measurable set $E \subset (0,x]$ be given. Choose $I \subset \T$ such that $|I| = 1-\sqrt{1-x}$, and let $A=\{z \in Q_I : 1-|z|^2 \in E\}$. Suppose $|E| < \alpha |(0,x]|$. By Lemma \ref{lem_dict}, $|A| < \alpha |Q_I|$, so $w(A) < \beta w(Q_I)$. Again by Lemma \ref{lem_dict}, $f(E) < \beta f((0,x])$. Thus, $f \in \PP 4$.
	
	$(\impliedby)$ Suppose $f \in \PP 4$. Let $I \subset \T$ and a measurable set $A \subset Q_I$ be given. Then $|A| \le |Q_I|$, so $|A|/|I| \le x$. Choose a measurable set $E \subset (0,x]$ such that $|E| = |A|/|I|$ and $\inf_E f \ge \sup_{(0,x] \setminus E} f$, and let $A' = \{z \in Q_I : 1-|z|^2 \in E\}$. Then $|A'|=|A|$ by Lemma \ref{lem_dict}, and $w(A') \ge w(A)$ since $\inf_{A'} w \ge \sup_{Q_I \setminus A'} w$. Suppose $|A| < \alpha |Q_I|$, i.e.\ $|A'| < \alpha |Q_I|$. By Lemma \ref{lem_dict}, $|E| < \alpha |(0,x]|$, so $f(E) < \beta f((0,x])$. Again by Lemma \ref{lem_dict}, $w(A') < \beta w(Q_I)$, so $w(A) < \beta w(Q_I)$. Thus, $w \in \PP 4$.
\end{proof}

The following two lemmas will be useful in Sections \ref{sec_mono} and \ref{sec_counter}, respectively.

\begin{lemma}
	\label{lem_AC}
	Let $w$ be a monotonic weight on $\D$ with $0<w(0)<\infty$. Suppose there exists $0<\varepsilon<1$ such that $w$ is almost constant on top halves $T_I$ with $|I|<\varepsilon$. Then $w$ is almost constant on all top halves.
\end{lemma}

\begin{proof}
	Assume $w$ is increasing. (The proof is similar if $w$ is decreasing.) Then $w(z)=f(|z|)$, where $f : [0,1) \to [0,\infty]$ is increasing. Since $f(0)>0$ and $w<\infty$ a.e., we have $0<f(t)<\infty$ for all $t\in[0,1)$. There exists $A_1>0$ such that $f(1-\frac{t}{2}) \le A_1 f(1-t)$ for all $0<t<\varepsilon$. Let $A_2 = f(1-\frac{\varepsilon}{2}) / f(0)$. Then $f(1-\frac{t}{2}) \le A_2 f(1-t)$ for all $\varepsilon \le t \le 1$, so $f(1-\frac{t}{2}) \le \max(A_1,A_2) f(1-t)$ for all $0<t\le1$.
\end{proof}

\begin{lemma}
	\label{lem_series}
	There exists $C>0$ such that, for all $n \ge 0$,
	\begin{equation}
		\label{eq_series}
		\sum_{k=n}^\infty 2^{-k} (k+1) \le C\,2^{-n} (n+1).
	\end{equation}
\end{lemma}

\begin{proof}
	Let $f(x) = 2^{-x}(x+1)$. Then $f'(x) = 2^{-x}[1-(x+1)\log2] < 0$ for $x\ge1$ (since $\log2>1/2$), so $f(x)$ is decreasing for $x\ge1$. Hence, for any $n\ge1$, we have
	\[
		\sum_{k=n+1}^\infty 2^{-k}(k+1)
		\le \int_n^\infty 2^{-x}(x+1)\,dx
		= \frac{2^{-n}(n+1)}{\log2} + \frac{2^{-n}}{\log^2 2}
		\le \frac{2}{\log2}\,2^{-n}(n+1).
	\]
	Here we used integration by parts and the fact that $1/\log2 < 2$. Thus, \eqref{eq_series} holds with $C = 1+(2/\log2)$ for $n\ge1$, so it holds with $C = 2+(2/\log2)$ for $n\ge0$.
\end{proof}

The following proposition clarifies the missing case in Theorem \ref{thm_mono_1}.

\begin{proposition}
	\label{prop_dec}
	Let $w$ be a decreasing weight on $\D$ with $0<w(0)<\infty$. Then $w \in B_\infty$.
\end{proposition}

\begin{proof}
	We have $w(z)=g(|z|)$, where $g:[0,1)\to(0,\infty)$ is decreasing. In particular, $w$ is bounded, hence integrable.  Let $f(t)=g(\sqrt{1-t})$ so that $w(z)=f(1-|z|^2)$ and $f:(0,1]\to(0,\infty)$ is increasing. Given $I\subset\T$, let $x=|I|(2-|I|)$. The function $y \mapsto \frac{1}{y} \int_0^y f$ is increasing, so $M(f1_{(0,x]})(t) = \frac{1}{x} \int_0^x f$ for all $t\in(0,x]$ by Lemma \ref{lem_dict} part (c). Hence $\int_0^x M(f1_{(0,x]}) = \int_0^x f$. By Lemma \ref{lem_dict} parts (a) and (d), it follows that $\int_{Q_I} M(w1_{Q_I}) = \int_{Q_I} w$.
\end{proof}

To deduce a self-improvement property for monotonic $B_p$ weights, we will need the following theorem, the proof of which is an adaptation of the proof of the self-improvement property of $A_p$ weights \cite{D}*{Corollary 7.6}.

\begin{theorem}
	\label{thm_self}
	Let $w$ be a weight on $\D$ which is almost constant on top halves.
	\begin{enumerate}[label=\upshape(\alph*)]
		\item If $w \in B_p$ for some $1<p<\infty$, then $w \in B_q$ for some $1<q<p$.
		\item If $w \in B_p$ for some $1 \le p < \infty$, then $w^{1+\varepsilon} \in B_p$ for some $\varepsilon>0$.
	\end{enumerate}
\end{theorem}

\begin{proof}
	(a) The weight $w^{1-p'}$ belongs to $B_{p'}$ and is almost constant on top halves. By Theorem \ref{thm_equiv_1}, this weight has the reverse H\"older property, i.e.\ there exist $r>1$ and $C>0$ such that
	\begin{equation}
		\label{eq_self_1}
		\left(\frac{1}{|Q_I|} \int_{Q_I} w^{(1-p')r}\right)^{1/r}
		\le \frac{C}{|Q_I|} \int_{Q_I} w^{1-p'}
	\end{equation}
	for all $I\subset\T$. Let $q = ((p-1)/r) + 1$. Then $1<q<p$ and $q' = (p'-1)r + 1$, so
	\[
		\left(\frac{1}{|Q_I|} \int_{Q_I} w^{1-q'}\right)^{q-1}
		\le \left(\frac{C}{|Q_I|} \int_{Q_I} w^{1-p'}\right)^{p-1}
	\]
	for all $I\subset\T$. Since $w \in B_p$, this implies that $w \in B_q$.
	
	(b) First suppose $p>1$. Then $w$ and $w^{1-p'}$ both have the reverse H\"older property by Theorem \ref{thm_equiv_1}, so there exist $r>1$ and $C>0$ such that
	\begin{equation}
		\label{eq_self_2}
		\left(\frac{1}{|Q_I|} \int_{Q_I} w^r\right)^{1/r}
		\le \frac{C}{|Q_I|} \int_{Q_I} w
	\end{equation}
	for all $I\subset\T$ and \eqref{eq_self_1} holds for all $I\subset\T$. Since $w \in B_p$, this implies that $w^r \in B_p$.
	
	Now, suppose $p=1$. Then $w$ has the reverse H\"older property by Theorem \ref{thm_equiv_1}, i.e.\ there exist $r>1$ and $C>0$ such that \eqref{eq_self_2} holds for all $I\subset\T$. Since $w \in B_1$ and
	\[
		(\essinf_{Q_I} w)^r = \essinf_{Q_I}(w^r)
	\]
	for all $I\subset\T$, this implies that $w^r \in B_1$.
\end{proof}

\section{Proof of Theorem \ref{thm_equiv_2}}
\label{sec_equiv}

Let us collect a few facts about arbitrary weights on $\D$.

\begin{proposition}
	\label{prop_arb_1}
	For arbitrary weights on $\D$, the following implications hold:
	\begin{enumerate}[label=\upshape(\alph*)]
		\item $\PP 1 \implies \PP 2 \implies \PP 5 \implies \PP 4$.
		\item $\PP 8 \implies \PP 3 \implies \PP 6 \implies \PP 4$.
		\item $\PP 4 \implies \PP 7$.
	\end{enumerate}
\end{proposition}

\begin{proof}
	Parts (a) and (b) follow immediately from Theorem 4.1 in \cite{DMO1}. Part (c) can be proved in the same way as Theorem 6.1 in \cite{DMO1}, which states that \ref{P4} implies \ref{P7} for the basis of Carleson cubes in the upper half-space. The proof of the two-dimensional case can easily be adapted to the basis of Carleson squares in the unit disc, thanks to Lemma \ref{lem_sup}.
\end{proof}

\begin{proposition}
	\label{prop_arb_2}
	For arbitrary weights on $\D$, the following inclusions hold:
	\begin{enumerate}[label=\upshape(\alph*)]
		\item For any $1<p<q<\infty$, we have $B_1 \subset B_p \subset B_q \subset RJ$.
		\item We have $RJ \subset B_\infty$ and $RH \subset B_\infty$.
	\end{enumerate}
\end{proposition}

\begin{proof}
	The first two inclusions in part (a) can be proved in the same way as the corresponding inclusions for $A_p$ weights (see e.g.\ \cite{D}*{Proposition 7.2}). The last inclusion in part (a) and the two inclusions in part (b) follow immediately from Proposition \ref{prop_arb_1}.
\end{proof}

The missing piece in the proof of Theorem \ref{thm_equiv_2} is provided by the following proposition, whose proof is inspired by that of Theorem IV in \cite{CF}.

\begin{proposition}
	\label{prop_P4P8}
	Let $w$ be a weight on $\D$ which is almost constant on top halves. If $w\in\PP4$, then $w\in\PP8$.
\end{proposition}

\begin{proof}
	Recall that \ref{P4} is equivalent to \ref{P4'}. Applying the function $t \mapsto \mu(B)-t$ to both sides of the inequality in \ref{P4'}, we see that there exist $\alpha,\beta\in(0,1)$ such that, for every $I\subset\T$,
	\begin{equation}
		\label{eq_P4P8_1}
		\alpha|Q_I| \le |\{z \in Q_I : w(z) > \beta w_{Q_I}\}|.
	\end{equation}
	Since $w \in AC$, there exists $C \ge 1$ such that $w(z) \le C w_{T_I}$ for every $I\subset\T$ and a.e.\ $z \in T_I$. For any $I\subset\T$, we have $w(T_I) \le w(Q_I)$ and $|Q_I| \le 4|T_I|$, so $w_{T_I} \le 4w_{Q_I}$ and hence $w(z) \le 4Cw_{Q_I}$ for a.e.\ $z \in T_I$.
	
	Given $I\subset\T$ and $\lambda>w_{Q_I}$, let $\DD(I)$ be the set of all dyadic descendants of $I$. Then $Q_I = \bigsqcup_{J \in \DD(I)} T_J$. For $z \in Q_I$, define the dyadic maximal function
	\[
		M_{I,d}w(z) = \sup_{J \in \DD(I)\,:\,z \in Q_J} w_{Q_J}.
	\]
	Note that $M_{I,d}w$ is constant on $T_J$ for all $J \in \DD(I)$. For each $J \in \DD(I)$ and a.e.\ $z \in T_J$, we have $w(z) \le 4Cw_{Q_J}$ and $w_{Q_J} \le M_{I,d}w(z)$, so $w(z) \le 4CM_{I,d}w(z)$. Consequently, this last inequality holds for a.e.\ $z \in Q_I$. It follows that
	\[
		w(\{z \in Q_I : w(z)>\lambda\}) \le w(\{z \in Q_I : M_{I,d}w(z)>\lambda'\}),
	\]
	where $\lambda'=\lambda/(4C)$.
	
	Let $\LL$ be the set of all maximal intervals $L \in \DD(I)$ such that $w_{Q_L}>\lambda'$. Then we have the following decomposition:
	\begin{equation}
		\label{eq_P4P8_2}
		\{z \in Q_I : M_{I,d}w(z)>\lambda'\} = \bigsqcup_{L \in \LL} Q_L.
	\end{equation}
	
	Case 1: Suppose $I \in \LL$. Then $\LL = \{I\}$, so \eqref{eq_P4P8_2} gives
	\begin{align*}
		w(\{z \in Q_I : M_{I,d}w(z)>\lambda'\})
		&= w(Q_I)
		\le \lambda |Q_I| \\
		&\le \lambda \alpha^{-1} |\{z \in Q_I : w(z) > \beta w_{Q_I}\}| \\
		&\le \lambda \alpha^{-1} |\{z \in Q_I : w(z) > \beta \lambda'\}| \\
		&= \lambda \alpha^{-1} |\{z \in Q_I : w(z) > \beta' \lambda\}|,
	\end{align*}
	where $\beta' = \beta/(4C)$. Here we used $\lambda > w_{Q_I}$, \eqref{eq_P4P8_1}, and $w_{Q_I} > \lambda'$.
	
	Case 2: Suppose $I \notin \LL$. Then, for every $L \in \LL$, the dyadic parent $K$ of $L$ does not belong to $\LL$, i.e.\ $w_{Q_K} \le \lambda'$. Since $w(Q_L) \le w(Q_K)$ and $|Q_K| \le 4|Q_L|$, we have $w_{Q_L} \le 4w_{Q_K}$, so $w_{Q_L} \le 4\lambda'$. By \eqref{eq_P4P8_2},
	\begin{align*}
		w(\{z \in Q_I : M_{I,d}w(z)>\lambda'\})
		&= \sum_{L \in \LL} w(Q_L)
		\le 4\lambda' \sum_{L \in \LL} |Q_L| \\
		&\le 4\lambda'\alpha^{-1} \sum_{L \in \LL} |\{z \in Q_L : w(z) > \beta w_{Q_L}\}| \\
		&\le 4\lambda'\alpha^{-1} |\{z \in Q_I : w(z) > \beta w_{Q_L}\}| \\
		&\le 4\lambda'\alpha^{-1} |\{z \in Q_I : w(z) > \beta\lambda'\}| \\
		&\le \lambda\alpha^{-1} |\{z \in Q_I : w(z)>\beta'\lambda\}|,
	\end{align*}
	where $\beta' = \beta/(4C)$. Here we used $w_{Q_L} \le 4\lambda'$, \eqref{eq_P4P8_1}, $w_{Q_L}>\lambda'$, and $C\ge1$.
\end{proof}

Now we are ready to prove Theorem \ref{thm_equiv_2}.

\begin{proof}[Proof of Theorem \textup{\ref{thm_equiv_2}}]
	Theorem \ref{thm_equiv_1} states that, for weights on $\D$ which are almost constant on top halves, \ref{P1}, \ref{P2}, \ref{P3} and \ref{P7} are equivalent. Parts (a) and (c) of Proposition \ref{prop_arb_1} allow us to add \ref{P4} and \ref{P5} to this list. Part (b) of Proposition \ref{prop_arb_1} and Proposition \ref{prop_P4P8} allow us to add \ref{P6} and \ref{P8} to this list.
\end{proof}

\section{Proof of Theorem \ref{thm_mono_1}}
\label{sec_mono}

Theorem \ref{thm_mono_1} is an immediate consequence of Proposition \ref{prop_arb_1} and Theorem \ref{thm_mono_2}, so it suffices to prove Theorem \ref{thm_mono_2}. We will need the following four lemmas.

\begin{lemma}
	\label{lem_P7_1}
	Let $u : (0, 1) \to (1, \infty)$ be a function such that $\lim_{t \to 0} u(t) = \infty$ and such that, for some $b>0$ and all $0<t<1$, we have $tu(t) \le b$. Then, for every $C > 0$, there exists $0 < \delta < 1$ such that $a \coloneq u(\delta)$ satisfies the following: For every function $f : (0, a) \to (0, \infty)$ which is integrable, decreasing and satisfies $f(1/2) = 1$, if $f(1) < \delta$, then
	\[
		\int_0^a \left( \frac{1}{y} \int_0^y f \right)\,dy > C \int_0^a f.
	\]
\end{lemma}

\begin{proof}
	Choose $0<\delta<1$ small enough that $a \coloneq u(\delta)$ satisfies $1+\log(2a) > 2(1+b)C$. Suppose $f:(0,a)\to(0,\infty)$ is integrable and decreasing, $f(1/2)=1$, and $f(1)<\delta$. Then the function $y \mapsto \frac{1}{y} \int_0^y f$ is decreasing, so
	\[
		\int_0^a \biggl( \frac{1}{y} \int_0^y f \biggr)\,dy
		\ge \int_0^{1/2} \biggl( \frac{1}{1/2} \int_0^{1/2} f \biggr)\,dy
		+ \int_{1/2}^a \biggl( \frac{1}{y} \int_0^{1/2} f \biggr)\,dy
		= (1+\log(2a)) \int_0^{1/2} f.
	\]
	Since $\int_0^{1/2} f \ge 1/2$ and $(a-1)f(1) < a\delta \le b$, we have
	\[
		\int_0^a f
		\le \int_0^{1/2} f + \frac{1}{2}\,f\left(\frac{1}{2}\right) + (a-1) f(1)
		\le 2(1+b) \int_0^{1/2} f.
	\]
	Together, these two estimates imply that
	\[
		\frac{\int_0^a (\frac{1}{y} \int_0^y f)\,dy}{\int_0^a f}
		\ge \frac{1+\log(2a)}{2(1+b)}
		> C. \qedhere
	\]
\end{proof}

\begin{lemma}
	\label{lem_P7_2}
	Let $f : (0, 1) \to (0, \infty)$ be integrable, decreasing and bounded below by some $b > 0$. Suppose there exists $C > 0$ such that, for all $0<x<1$,
	\[
		\int_0^x \left( \frac{1}{y} \int_0^y f \right)\,dy \le C \int_0^x f.
	\]
	Then there exists $A > 0$ such that $f(x/2) \le A f(x)$ for all $0 < x < 1$.
\end{lemma}

\begin{proof}
	Suppose not. Then, for every $n \in \mathbb N$, there exists $0 < x_n < 1$ such that $f(x_n/2) > A_n f(x_n)$, where $A_n = 2^n b^{-1} f(2^{-n-1})$. We have $f(x/2) \le f(2^{-n-1}) = 2^{-n} A_n b \le A_n f(x)$ for all $x \ge 2^{-n}$, so $x_n < 2^{-n}$. Note that $1 < A_1 < A_2 < \cdots$ and $A_n \to \infty$. Hence we may define $u : (0, 1] \to [1, \infty)$ as follows: Set $u(1) = 1$ and $u(A_n^{-1}) = 2^n$ for all $n \ge 1$, and let $u$ be linear on each intervening interval. Then $u$ is strictly decreasing and $\lim_{t \to 0} u(t) = \infty$. Note the following:
	\begin{enumerate}[label=(\roman*)]
		\item For $A_1^{-1} \le t \le 1$, we have $tu(t) \le u(A_1^{-1}) = 2$.
		\item For $A_{n+1}^{-1} \le t \le A_n^{-1}$, we have $tu(t) \le A_n^{-1} u(A_{n+1}^{-1}) = 2b/f(2^{-n-1}) \le 2$.
	\end{enumerate}
	Thus, $tu(t) \le 2$ for all $0 < t \le 1$.
	
	Now, let $\delta$ be as given by Lemma \ref{lem_P7_1}, and let $a = u(\delta)$. Choose $n \in \mathbb N$ large enough that $A_n > 1/\delta$. Let $g(t) = f(x_n t)/f(x_n/2)$ for $0 < t < a$. (Note that $x_n a < 1$ since $x_n < 2^{-n}$ and $a < u(A_n^{-1}) = 2^n$.) Then $g(1) < A_n^{-1} < \delta$, so
	\[
		C < \frac{\int_0^a (\frac{1}{y} \int_0^y g)\,dy}{\int_0^a g}
		= \frac{\int_0^a (\frac{1}{x_ny} \int_0^{x_ny} f)\,dy}{\frac{1}{x_n} \int_0^{x_na} f}
		= \frac{\int_0^{x_na} (\frac{1}{y} \int_0^y f)\,dy}{\int_0^{x_na} f}.
	\]
	This is a contradiction.
\end{proof}

\begin{lemma}
	\label{lem_P4_1}
	For every $\alpha,\beta \in (0,1)$, there exists $0<\delta<1$ such that, for every function $f : (0, 2) \to (0, \infty)$ which is increasing and satisfies $f(1) = 1$, if $f(1/2) < \delta$, then there exists $0<x<2$ such that
	\[
		\int_{(1-\alpha)x}^x f > \beta \int_0^x f.
	\]
\end{lemma}

\begin{proof}
	Choose $n \in \mathbb N$ large enough that $2/n \le \alpha$, and choose $0<\delta<1$ small enough that $1+2n\sqrt[n]{\delta} < 1/\beta$. Suppose $f : (0,2) \to (0,\infty)$ is increasing, $f(1)=1$, and $f(1/2)<\delta$. For $i = 0, 1, \dots, n+1$, let $t_i = (n+i)/(2n)$. Then $\prod_{i=1}^n [f(t_i)/f(t_{i-1})] = f(t_n)/f(t_0) > 1/\delta$, so there exists $1 \le i \le n$ such that $f(t_i)/f(t_{i-1}) > 1/\sqrt[n]{\delta}$. Let $x = t_{i+1}$ and $y = t_{i-1}$. Then $1/2 \le y < x \le 3/2$ and $x-y = 1/n \le \alpha x$, so $\int_y^x f \le \int_{(1-\alpha)x}^x f$. Note that $y<1$. We have
	\[
		\int_0^y f
		\le yf(y)
		\le \sqrt[n]{\delta} f(t_i)
		\le \frac{\sqrt[n]{\delta}}{x-t_i} \int_{t_i}^x f
		\le 2n \sqrt[n]{\delta} \int_y^x f,
	\]
	so $\int_0^x f \le (1+2n\sqrt[n]{\delta}) \int_y^x f$. It follows that
	\[
		\frac{\int_{(1-\alpha)x}^x f}{\int_0^x f}
		\ge \frac{1}{1+2n\sqrt[n]{\delta}}
		> \beta. \qedhere
	\]
\end{proof}

\begin{lemma}
	\label{lem_P4_2}
	Let $f : (0, 1) \to (0, \infty)$ be increasing. Suppose there exist $\alpha,\beta \in (0,1)$ such that, for all $0<x<1$,
	\[
		\int_{(1-\alpha)x}^x f
		\le \beta \int_0^x f.
	\]
	Then there exists $A > 0$ such that $f(x) \le A f(x/2)$ for all $0 < x < 1/2$.
\end{lemma}

\begin{proof}
	Let $\delta$ be as given by Lemma \ref{lem_P4_1}. Given $0 < x < 1/2$, let $g(t) = f(xt)/f(x)$ for $0 < t < 2$. Then, for all $0<y<2$,
	\[
		\frac{\int_{(1-\alpha)y}^y g}{\int_0^y g}
		= \frac{\int_{(1-\alpha)xy}^{xy} f}{\int_0^{xy} f}
		\le \beta,
	\]
	so $g(1/2) \ge \delta$, i.e.\ $f(x) \le (1/\delta) f(x/2)$.
\end{proof}

Now we are ready to prove Theorem \ref{thm_mono_2}.

\begin{proof}[Proof of Theorem \textup{\ref{thm_mono_2}}]
	Since $w$ is monotonic and $0<w(0)<\infty$, we have $w(z) = g(|z|)$, where $g : [0,1) \to (0,\infty)$ is monotonic. Let $f(t) = g(\sqrt{1-t})$. Then $f : (0,1] \to (0,\infty)$ is monotonic and $g(r) = f(1-r^2)$, so $w(z) = f(1-|z|^2)$.
	
	Case 1: Suppose $w$ is increasing and $w\in\PP7$. By Lemma \ref{lem_equiv}, $f\in\PP7$, i.e.\ there exists $C>0$ such that $\int_0^x M(f1_{(0,x]}) \le C \int_0^x f$ for all $x \in (0,1]$. Since $f$ is decreasing, so is the function $y \mapsto \frac{1}{y} \int_0^y f$. By Lemma \ref{lem_dict} part (c), for all $x \in (0,1]$,
	\[
		\int_0^x \left(\frac{1}{t} \int_0^t f\right)\,dt \le C \int_0^x f.
	\]
	For $0<t<1$, let $h(t) = f(t/2)$. Then, for all $0<x<1$,
	\[
		\int_0^x \left( \frac{1}{y} \int_0^y h \right)\,dy
		\le C \int_0^x h.
	\]
	The function $h$ is bounded below by a positive number, so Lemma \ref{lem_P7_2} gives $A>0$ such that $h(t/2) \le A h(t)$ for all $0<t<1$. Then $f(t/2) \le A f(t)$ for all $0<t<1/2$. For all $0<r<1/4$, we have $0 < r(2-r) < 1/2$ and hence
	\[
		g\left(1-\frac{r}{2}\right)
		= f\left(r\left(1-\frac{r}{4}\right)\right)
		\le f\left(r\left(1-\frac{r}{2}\right)\right)
		\le A f(r(2-r))
		= A g(1-r).
	\]
	Thus, $w$ is almost constant on top halves $T_I$ with $|I| < 1/4$. By Lemma \ref{lem_AC}, $w$ is almost constant on all top halves.
	
	Case 2: Suppose $w$ is decreasing and $w \in \PP 4$. By Lemma \ref{lem_equiv}, $f \in \PP 4$, i.e.\ there exist $\alpha,\beta \in (0,1)$ such that, for any $x \in (0,1]$ and any measurable set $E \subset (0,x]$, if $|E| \le \alpha x$, then $\int_E f \le \beta \int_0^x f$. In particular, for all $x \in (0,1]$,
	\[
		\int_{(1-\alpha)x}^x f \le \beta \int_0^x f.
	\]
	Since $f$ is increasing, Lemma \ref{lem_P4_2} gives $A>0$ such that $f(x) \le A f(x/2)$ for all $0<x<1/2$. For all $0<r<1/4$, we have $0<r(2-r)<1/2$ and hence
	\[
		g(1-r)
		= f(r(2-r))
		\le A f\left(r\left(1-\frac{r}{2}\right)\right)
		\le A f\left(r\left(1-\frac{r}{4}\right)\right)
		= A g\left(1-\frac{r}{2}\right).
	\]
	Thus, $w$ is almost constant on top halves $T_I$ with $|I| < 1/4$. By Lemma \ref{lem_AC}, $w$ is almost constant on all top halves.
\end{proof}

We conclude this section with a few corollaries of Theorems \ref{thm_mono_1} and \ref{thm_mono_2}.

\begin{corollary}
	\label{cor_1}
	Let $w$ be a weight on $\D$ such that $0<w(0)<\infty$.
	\begin{enumerate}[label=\upshape(\alph*)]
		\item If $w$ is increasing, then \textup{\ref{P1}} - \textup{\ref{P8}} are equivalent.
		\item If $w$ is decreasing, then \textup{\ref{P1}} - \textup{\ref{P6}}, \textup{\ref{P8}}, and the condition $w \in AC$ are equivalent.
	\end{enumerate}
\end{corollary}

\begin{proof}
	(a) If $w$ satisfies one of \ref{P1} - \ref{P8}, then $w\in\PP7$ by Proposition \ref{prop_arb_1}, so $w \in AC$ by Theorem \ref{thm_mono_2}, and it follows by Theorem \ref{thm_equiv_2} that $w$ satisfies all of \ref{P1} - \ref{P8}.
	
	(b) If $w$ satisfies one of \ref{P1} - \ref{P6} or \ref{P8}, then $w\in\PP4$ by Proposition \ref{prop_arb_1}, so $w \in AC$ by Theorem \ref{thm_mono_2}. Conversely, if $w \in AC$, then, since $w\in\PP7$ by Proposition \ref{prop_dec}, it follows by Theorem \ref{thm_equiv_2} that $w$ satisfies all of \ref{P1} - \ref{P8}.
\end{proof}

\begin{remark}
	\label{rmk_cor}
	Example \ref{ex_10} shows that the condition $w \in AC$ cannot be included among the equivalent conditions in Corollary \ref{cor_1}(a). Example \ref{ex_8} shows that \ref{P7} cannot be included among the equivalent conditions in Corollary \ref{cor_1}(b).
\end{remark}

\begin{corollary}
	\label{cor_2}
	Let $w$ be a weight on $\D$ such that $0<w(0)<\infty$. Consider the following conditions:
	\begin{enumerate}[label=\upshape(\alph*)]
		\item $w$ is almost constant on top halves.
		\item $w$ has the reverse H\"older property.
		\item $w$ has the reverse Jensen property.
		\item $w \in B_p$ for some $1<p<\infty$.
		\item $w \in B_\infty$.
	\end{enumerate}
	If $w$ is increasing, then conditions \textup{(b)} - \textup{(e)} are equivalent. If $w$ is decreasing, then conditions \textup{(a)} - \textup{(d)} are equivalent.
\end{corollary}

\begin{proof}
	This follows immediately from Corollary \ref{cor_1}. Alternatively, it can easily be deduced from Theorems \ref{thm_equiv_1} and \ref{thm_mono_1} and Proposition \ref{prop_dec}.
\end{proof}

\begin{corollary}
	\label{cor_3}
	Let $w$ be a monotonic weight on $\D$ such that $0<w(0)<\infty$.
	\begin{enumerate}[label=\upshape(\alph*)]
		\item If $w \in B_p$ for some $1<p<\infty$, then $w \in B_q$ for some $1<q<p$.
		\item If $w \in B_p$ for some $1 \le p < \infty$, then $w^{1+\varepsilon} \in B_p$ for some $\varepsilon>0$.
	\end{enumerate}
\end{corollary}

\begin{proof}
	This follows immediately from Theorems \ref{thm_mono_1} and \ref{thm_self}.
\end{proof}

\section{Counterexamples}
\label{sec_counter}

We begin by constructing a family of radial weights depending on a choice of sequences $(a_k)$ and $(b_k)$ of real numbers.

\begin{example}
	\label{ex_main}
	Let $(a_k)_{k=0}^\infty \subset [0,1]$ and $(b_k)_{k=0}^\infty \subset (0,\infty)$. Define $f:(0,1]\to(0,\infty)$ by
	\[
		f(t) =
		\begin{cases}
			1 & \text{if } 2^{-k-1}(1+a_k)<t\le2^{-k} \\
			b_k & \text{if } 2^{-k-1}<t\le2^{-k-1}(1+a_k)
		\end{cases}
	\]
	for all $k \ge 0$. Define $w:\D\to(0,\infty)$ by $w(z) = f(1-|z|^2)$. Note that, if $0 < a_k < 1$ for all $k \ge 0$ and either $b_k \to \infty$ or $b_k \to 0$, then $w \notin AC$.
	
	The following estimates will be useful: Suppose $\varphi : (0,\infty)\to\R$ is a measurable function such that $\int_0^1 (\varphi \circ f)$ is defined. Then, for any $n \ge 0$, we have
	\[
		\int_0^{2^{-n}} (\varphi \circ f)
		= \sum_{k=n}^\infty \int_{2^{-k-1}}^{2^{-k}} (\varphi \circ f)
		= \sum_{k=n}^\infty 2^{-k-1}(a_k \varphi(b_k) + (1-a_k) \varphi(1)).
	\]
	For any $0 < x \le 1$, we have $2^{-n-1} < x \le 2^{-n}$ for some $n \ge 0$. If $\varphi(1) \ge 0$ and $\varphi(b_k) \ge 0$ for all $k\ge0$, then
	\begin{align}
		\frac{1}{x} \int_0^x (\varphi \circ f)
		&\le 2^{n+1} \int_0^{2^{-n}} (\varphi \circ f)
		\le 2\varphi(1) + \sum_{k=n}^\infty 2^{n-k} a_k \varphi(b_k), \label{eq_1} \\
		\frac{1}{x} \int_0^x (\varphi \circ f)
		&\ge 2^n \int_0^{2^{-n-1}} (\varphi \circ f)
		\ge \frac{1-A}{2}\,\varphi(1) + \frac{1}{4} \sum_{k=n+1}^{\infty} 2^{n+1-k} a_k \varphi(b_k), \label{eq_2}
	\end{align}
	where $A = \sup_{k\ge1} a_k$. In particular, $\int_0^1 f \le 2 + \sum_{k=0}^\infty 2^{-k} a_k b_k$, so, if there is a polynomial $p(x)$ such that $a_k b_k \le p(k)$ for all $k \ge 0$, then $f$ is integrable and, by Lemma \ref{lem_dict} part (a), $w$ is integrable.
\end{example}

\begin{remark}
	\label{rmk_smooth}
	Example \ref{ex_main} can easily be modified to obtain a smooth weight $w$ satisfying similar estimates. For example, one can redefine $f$ on each of the intervals ${[2^{-k-1},2^{-k-1}(1+a_k)]}$ so that it is smooth on this interval, equal to $1$ on the end sixths of the interval, and equal to $b_k$ on the middle third of the interval.
\end{remark}

Now, we show how the sequences $(a_k)$ and $(b_k)$ can be chosen to obtain the desired counterexamples. In Examples \ref{ex_1} - \ref{ex_7}, we have $a_k b_k \le k+1$ for all $k\ge0$, so $w$ is integrable.

\begin{example}[$w\in\PP{1}\setminus\PP{6}$ and $w \in B_1 \setminus RH$]
	\label{ex_1}
	In Example \ref{ex_main}, let $b_0 \ge 1$, $b_k \nearrow \infty$, and $a_k = 1/b_k$, so that $a_0 \le 1$ and $a_k \searrow 0$.
	
	Proof of $w\in\PP{1}$ and $w \in B_1$: Given $x \in (0,1]$, we have $2^{-n-1} < x \le 2^{-n}$ for some $n\ge0$. Since $a_k b_k \le 1$ for all $k\ge0$, \eqref{eq_1} gives
	\[
		\frac{1}{x} \int_0^x f
		\le 2 + \sum_{k=n}^\infty 2^{n-k} a_k b_k
		\le 4.
	\]
	Since $f \ge 1$, we have $\essinf_{(0,x]} f \ge 1$. Thus, $f \in A_1$. By Lemma \ref{lem_equiv}, $w \in B_1$. By Proposition \ref{prop_arb_2}, $w\in\PP{1}$.
	
	Proof of $w\notin\PP{6}$ and $w \notin RH$: Given $n \ge 0$, let $x = 2^{-n}$. Since $\frac{1}{x} \int_0^x f \le 4$, \eqref{eq_2} gives
	\begin{align*}
		\frac{1}{x} \int_0^x f \log^+ \biggl(\frac{f}{\frac{1}{x} \int_0^x f}\biggr)
		&\ge \frac{1}{x} \int_0^x f \log^+ \biggl(\frac{f}{4}\biggr)
		\ge \frac{a_{n+1} b_{n+1} \log^+(b_{n+1}/4)}{4} \\
		&= \frac{\log^+(b_{n+1}/4)}{4}
		\xrightarrow[n\to\infty]{} \infty.
	\end{align*}
	Thus, $f \notin \PP 6$. By Lemma \ref{lem_equiv}, $w \notin \PP 6$. By Proposition \ref{prop_arb_1}, $w \notin RH$.
\end{example}

\begin{example}[$w\in\PP{8}\setminus\PP{5}$ and $w \in RH \setminus RJ$]
	\label{ex_2}
	In Example \ref{ex_main}, let $b_0 = 1$, $b_k \nearrow \infty$, and $b_k/b_{k-1} \le (k+1)/k$ (e.g.\ $b_k = k+1$), and let $a_k = 1/4$. By induction, $b_k \le k+1$.
	
	Proof of $w\in\PP{8}$ and $w \in RH$: Let $x \in (0,1]$ and $\lambda > \frac{1}{x} \int_0^x f$ be given. Then $2^{-n-1} < x \le 2^{-n}$ for some $n \ge 0$. Since $f \ge 1$, we have $\frac{1}{x} \int_0^x f \ge 1$. This implies that $\lambda > 1$, so $b_m \le \lambda < b_{m+1}$ for some $m \ge 0$. Let $N = \max(m,n)$. By \eqref{eq_2},
	\[
		b_{m+1}
		> \frac{1}{x} \int_0^x f
		\ge \frac{a_{n+1} b_{n+1}}{4}
		= \frac{b_{n+1}}{16}.
	\]
	Hence $b_{N+1} \le 16 b_{m+1}$. Let $C$ be as given by Lemma \ref{lem_series}. For any $k \ge N$, we have
	\[
		\frac{b_k}{b_N}
		= \prod_{j=N+1}^k \frac{b_j}{b_{j-1}}
		\le \prod_{j=N+1}^k \frac{j+1}{j}
		= \frac{k+1}{N+1}.
	\]
	For any $M\ge0$, let $E_M = \bigsqcup_{k=M}^\infty (2^{-k-1},2^{-k-1}(1+a_k)]$. Note that
	\begin{align*}
		(0,x] \cap \{f>\lambda\}
		&\subset (0,2^{-n}] \cap \{f>b_m\}
		\subset (0,2^{-N}] \cap \{f>1\}
		\subset E_N, \\
		(0,x] \cap \{f>\lambda\}
		&\supset (0,2^{-n-1}] \cap \{f \ge b_{m+1}\}
		\supset (0,2^{-N-1}] \cap \{f \ge b_{N+1}\}
		\supset E_{N+1}.
	\end{align*}
	Using all this information, we estimate
	\begin{align*}
		f((0,x] \cap \{f>\lambda\})
		&\le \sum_{k=N}^\infty 2^{-k-1} a_k b_k
		\le \frac{b_N}{8(N+1)} \sum_{k=N}^\infty 2^{-k}(k+1)
		\le \frac{C}{8} \cdot 2^{-N} b_N, \\
		|(0,x] \cap \{f>\lambda\}|
		&\ge \sum_{k=N+1}^\infty 2^{-k-1} a_k
		= \frac{1}{8} \cdot 2^{-N}.
	\end{align*}
	It follows that
	\[
		\frac{f((0,x] \cap \{f>\lambda\})}{\lambda|(0,x] \cap \{f>\lambda\}|}
		\le \frac{Cb_N}{\lambda}
		\le \frac{16Cb_{m+1}}{b_m}
		\le \frac{16C(m+2)}{m+1}
		\le 32C.
	\]
	Thus, $f\in\PP{8}$. By Lemma \ref{lem_equiv}, $w\in\PP{8}$. By Proposition \ref{prop_arb_1}, $w \in RH$.
	
	Proof of $w\notin\PP{5}$ and $w \notin RJ$: Given $n \ge 0$, let $x = 2^{-n}$. Since $f=1$ on at least $3/4$ of $(2^{-k-1},2^{-k}]$ for all $k \ge 0$, we have $f=1$ on at least $3/4$ of $(0,2^{-n}]$, so $m(f;(0,x]) = 1$. However, $\frac{1}{x} \int_0^x f \ge b_{n+1}/16 \to \infty$ as $n\to\infty$. Thus, $f\notin\PP{5}$. By Lemma \ref{lem_equiv}, $w\notin\PP{5}$. By Proposition \ref{prop_arb_1}, $w \notin RJ$.
\end{example}

\begin{example}[$w\in\PP{8}\cap\PP{5}\setminus\PP{2}$ and $w \in RH \setminus RJ$]
	\label{ex_3}
	In Example \ref{ex_main}, let $b_0 = 1$, $b_k \searrow 0$, and $a_k = 1/\sqrt{16-\log b_k}$, so that $a_0 = 1/4$, $a_k \searrow 0$, and $b_k = \exp(16-(1/a_k^2))$.
	
	Proof of $w\in\PP{8}$ and $w \in RH$: Let $x\in(0,1]$ and $\lambda > \frac{1}{x} \int_0^x f$ be given. Then $2^{-n-1} < x \le 2^{-n}$ for some $n \ge 0$. If $\lambda \ge 1$, then $\{f>\lambda\} = \varnothing$ (since $f \le 1$). Suppose $\lambda<1$. Since $f \le 1$, we have $f((0,x] \cap \{f>\lambda\}) \le f((0,2^{-n}]) \le 2^{-n}$. Since $\sup_{k \ge 1} a_k \le 1/2$, we have $\frac{1}{x} \int_0^x f \ge 1/4$ by \eqref{eq_2}, so $\lambda>1/4$. Since
	\[
		(0,x] \cap \{f>\lambda\}
		\supset (0,2^{-n-1}] \cap \{f=1\}
		\supset \bigsqcup_{k=n+1}^\infty (2^{-k-1}(1+a_k),2^{-k}],
	\]
	we have $|(0,x] \cap \{f>\lambda\}| \ge \sum_{k=n+1}^\infty 2^{-k-1} (1-a_k) \ge (1/4) \cdot 2^{-n}$. Thus, $f\in\PP{8}$ with $\beta=1$ and $C=16$. By Lemma \ref{lem_equiv}, $w\in\PP{8}$. By Proposition \ref{prop_arb_1}, $w \in RH$.
	
	Proof of $w\in\PP{5}$: For each $k\ge0$, we have $a_k \le 1/4$, so $f=1$ on at least $3/4$ of $(2^{-k-1},2^{-k}]$. Hence, for each $n\ge0$, we have $f=1$ on at least $3/4$ of $(0,2^{-n}]$. Given $x \in (0,1]$, we have $2^{-n-1} < x \le 2^{-n}$ for some $n \ge 0$. We claim that $f=1$ on at least $3/5$ of $(0,x]$. Indeed, in the worst case, we have $x = 2^{-n-1}(1+a_n)$, and the proportion of the interval $(0,x]$ on which $f=1$ is at least
	\[
		\frac{(3/4)\cdot2^{-n-1}}{2^{-n-1}(1+a_n)}
		\ge \frac{(3/4)}{1+(1/4)}
		= \frac{3}{5}.
	\]
	It follows that $m(f;(0,x]) = 1$. Since $f \le 1$, we have $\frac{1}{x} \int_0^x f \le 1$. Thus, $f\in\PP{5}$. By Lemma \ref{lem_equiv}, $w\in\PP{5}$.
	
	Proof of $w\notin\PP{2}$, i.e.\ $w \notin RJ$: Given $n \ge 0$, let $x=2^{-n}$. Since $f \le 1$, we have $\frac{1}{x} \int_0^x \log(f^{-1}) \ge a_{n+1} \log(b_{n+1}^{-1}) / 4$ by \eqref{eq_2}. Since $\frac{1}{x} \int_0^x f \ge 1/4$, it follows that
	\[
		\left(\frac{1}{x} \int_0^x f\right) \exp\left(\frac{1}{x} \int_0^x \log(f^{-1})\right)
		\ge \frac{1}{4} \exp\left(\frac{1}{4a_{n+1}}-4a_{n+1}\right)
		\xrightarrow[n\to\infty]{} \infty.
	\]
	Thus, $f\notin\PP{2}$. By Lemma \ref{lem_equiv}, $w\notin\PP{2}$.
\end{example}

\begin{example}[$w\in\PP{8}\cap\PP{2}\setminus\PP{1}$ and $w \in RH \cap RJ \setminus (\bigcup_{1<p<\infty} B_p)$]
	\label{ex_4}
	In Example \ref{ex_main}, let $b_0 = 1$, $b_k \searrow 0$, and $a_k = 1/(4-\log{b_k})$, so that $a_0 = 1/4$, $a_k \searrow 0$, and $b_k = \exp(4-(1/a_k))$.
	
	Proof of $w\in\PP{8}$ and $w \in RH$: This is the same as in Example \ref{ex_3}.
	
	Proof of $w\in\PP{2}$, i.e.\ $w \in RJ$: Given $x \in (0,1]$, we have $2^{-n-1} < x \le 2^{-n}$ for some $n\ge0$. Since $f \le 1$, we have $\frac{1}{x} \int_0^x f \le 1$. Since $a_k \log(b_k^{-1}) \le 1$ for all $k\ge0$, \eqref{eq_1} gives
	\[
		\frac{1}{x} \int_0^x \log(f^{-1})
		\le \sum_{k=n}^\infty 2^{n-k} a_k \log(b_k^{-1})
		\le 2.
	\]
	Thus, $f\in\PP{2}$. By Lemma \ref{lem_equiv}, $w\in\PP{2}$.
	
	Proof of $w\notin\PP{1}$, i.e.\ $w \notin \bigcup_{1<p<\infty} B_p$: Given $1<p<\infty$ and $n\ge0$, let $x=2^{-n}$. We have $\frac{1}{x} \int_0^x f \ge 1/4$ and $\frac{1}{x} \int_0^x f^{1-p'} \ge a_{n+1} b_{n+1}^{1-p'} / 4$ by \eqref{eq_2}, so
	\[
		\left(\frac{1}{x} \int_0^x f\right) \left(\frac{1}{x} \int_0^x f^{1-p'}\right)^{p-1}
		\ge \frac{\exp((1/a_{n+1})-4)}{4^p (1/a_{n+1})^{p-1}}
		\xrightarrow[n\to\infty]{} \infty.
	\]
	Thus, $f\notin\PP{1}$. By Lemma \ref{lem_equiv}, $w\notin\PP{1}$.
\end{example}

\begin{example}[$w\in\PP{1}\cap\PP{6}\setminus\PP{3}$ and $w \in B_1 \setminus RH$]
	\label{ex_5}
	In Example \ref{ex_main}, let $b_0 \ge 1$, $b_k \nearrow \infty$, and $a_k = 1/(b_k(1+\log b_k))$, so that $a_0 \le 1$ and $a_k \searrow 0$.
	
	Proof of $w\in\PP{1}$ and $w \in B_1$: This is the same as in Example \ref{ex_1}.
	
	Proof of $w\in\PP{6}$: Given $x \in (0,1]$, we have $2^{-n-1} < x \le 2^{-n}$ for some $n \ge 0$. Since $f \ge 1$, we have $\frac{1}{x} \int_0^x f \ge 1$. By \eqref{eq_1},
	\begin{align*}
		\frac{1}{x} \int_0^x f \log^+\biggl(\frac{f}{\frac{1}{x} \int_0^x f}\biggr)
		&\le \frac{1}{x} \int_0^x f \log f
		\le \sum_{k=n}^\infty 2^{n-k} a_k b_k \log b_k
		\le \sum_{k=n}^\infty 2^{n-k}
		= 2.
	\end{align*}
	Thus, $f\in\PP{6}$. By Lemma \ref{lem_equiv}, $w\in\PP{6}$.
	
	Proof of $w\notin\PP{3}$, i.e.\ $w \notin RH$: Given $1<q<\infty$ and $n\ge0$, let $x=2^{-n}$. By \eqref{eq_2}, $\frac{1}{x} \int_0^x f^q \ge a_{n+1} b_{n+1}^q / 4$. Since $\frac{1}{x} \int_0^x f \le 4$, it follows that
	\[
		\frac{(\frac{1}{x} \int_0^x f^q)^{1/q}}{\frac{1}{x} \int_0^x f}
		\ge \frac{b_{n+1}^{1-(1/q)}}{4^{1+(1/q)}(1+\log b_{n+1})^{1/q}}
		\xrightarrow[n\to\infty]{} \infty.
	\]
	Thus, $f\notin\PP{3}$. By Lemma \ref{lem_equiv}, $w\notin\PP{3}$.
\end{example}

\begin{example}[$w\in\PP{1}\cap\PP{3}\setminus\PP{8}$ and $w \in B_1 \cap RH \setminus AC$]
	\label{ex_6}
	In Example \ref{ex_main}, let $b_0 \ge 1$, $b_k \nearrow \infty$, and $a_k = 2^{-b_k}$, so that $a_0 \le 1/2$ and $a_k \searrow 0$.
	
	Proof of $w\in\PP{1}$ and $w \in B_1$: This is the same as in Example \ref{ex_1}.
	
	Proof of $w\in\PP{3}$, i.e.\ $w \in RH$: For any $1<q<\infty$, we have $a_k b_k^q \to 0$ as $k\to\infty$, so there exists $C_q>0$ such that $a_k b_k^q \le C_q$ for all $k \ge 0$. Given $x \in (0,1]$, we have $2^{-n-1} < x \le 2^{-n}$ for some $n \ge 0$. By \eqref{eq_1},
	\[
		\frac{1}{x} \int_0^x f^q
		\le 2 + \sum_{k=n}^\infty 2^{n-k} a_k b_k^q
		\le 2(C_q+1).
	\]
	Since $f \ge 1$, we have $\frac{1}{x} \int_0^x f \ge 1$. Thus, $f\in\PP{3}$. By Lemma \ref{lem_equiv}, $w\in\PP{3}$.
	
	Proof of $w\notin\PP{8}$: Given $\beta,C>0$, choose $n$ large enough that $\sqrt{b_n} > \max(4,1/\beta,2C)$ and let $x=2^{-n}$ and $\lambda=\sqrt{b_n}$. Then $\frac{1}{x} \int_0^x f \le 4 < \lambda$ and $\beta \lambda > 1$. Since $\lambda>1$, we have $b_n>\lambda$. Let $E_n = \bigsqcup_{k=n}^\infty (2^{-k-1},2^{-k-1}(1+a_k)]$. Then
	\begin{align*}
		(0,x] \cap \{f>\lambda\}
		&\supset (0,2^{-n}] \cap \{f \ge b_n\}
		\supset E_n, \\
		(0,x] \cap \{f>\beta\lambda\}
		&\subset (0,2^{-n}] \cap \{f>1\}
		\subset E_n.
	\end{align*}
	Hence
	\begin{align*}
		f((0,x] \cap \{f>\lambda\})
		&\ge \sum_{k=n}^\infty 2^{-k-1} a_k b_k
		\ge 2^{-n-1} a_n b_n, \\
		|(0,x] \cap \{f>\beta\lambda\}|
		&\le \sum_{k=n}^\infty 2^{-k-1} a_k
		\le a_n \sum_{k=n}^\infty 2^{-k-1}
		= 2^{-n} a_n.
	\end{align*}
	It follows that
	\[
		\frac{f((0,x] \cap \{f>\lambda\})}{\lambda |(0,x] \cap \{f>\beta\lambda\}|}
		\ge \frac{\sqrt{b_n}}{2}
		> C.
	\]
	Thus, $f\notin\PP{8}$. By Lemma \ref{lem_equiv}, $w\notin\PP{8}$.
\end{example}

\begin{example}[$w\in\PP{1}\cap\PP{8}\setminus AC$ and $w \in RH \cap (\bigcap_{1<p<\infty}{B_p}) \setminus B_1$]
	\label{ex_7}
	In Example \ref{ex_main}, let $b_0 \le 1$, $b_k \searrow 0$, and $a_k = 2^{-1/b_k}$, so that $a_0 \le 1/2$ and $a_k \searrow 0$.
	
	Proof of $w\in\PP{8}$ and $w \in RH$: This is the same as in Example \ref{ex_3}.
	
	Proof of $w\in\PP{1}$ and $w \in \bigcap_{1<p<\infty} B_p$: Since $f \le 1$, we have $\frac{1}{x} \int_0^x f \le 1$ for all $x \in (0,1]$. For any $q>0$, we have $a_k b_k^{-q} \to 0$ as $k\to\infty$, so there exists $C_q>0$ such that $a_k b_k^{-q} \le C_q$ for all $k \ge 0$. By \eqref{eq_1}, $\frac{1}{x} \int_0^x f^{-q} \le 2(C_q+1)$ for all $x\in(0,1]$. Thus, $f \in A_p$ for all $1<p<\infty$. By Lemma \ref{lem_equiv}, $w \in B_p$ for all $1<p<\infty$. In particular, $w\in\PP{1}$.
	
	Proof of $w \notin B_1$: Since $\essinf_{(0,1]} f = 0$, we have $f \notin A_1$. By Lemma \ref{lem_equiv}, $w \notin B_1$.
\end{example}

\begin{remark}
	\label{rmk_AC}
	Example \ref{ex_7} shows that, even if a weight $w$ on $\D$ satisfies all of the properties \ref{P1} - \ref{P8}, it still need not be almost constant on top halves.
\end{remark}

In the remainder of this section, we present counterexamples which are not based on the construction in Example \ref{ex_main}.

\begin{example}[$w\in\PP{7}\setminus\PP{4}$]
	\label{ex_8}
	Let $(b_k)_{k=0}^\infty$ be a sequence of real numbers such that $b_0=1$ and $b_k\nearrow\infty$. Let $w(z)=f(1-|z|)$, where $f(t)=1/(\prod_{k=1}^n b_k)$ for $2^{-n-1} < t \le 2^{-n}$ and $n\ge0$. Then $w$ is decreasing, $w(0)=1$, and $w \notin AC$, so $w \in \PP{7}$ by Proposition \ref{prop_dec}, but $w \notin \PP{4}$ by Theorem \ref{thm_mono_2}.
\end{example}

The following remark contains one of the key results of this paper.

\begin{remark}
	\label{rmk_osc}
	In Examples \ref{ex_1} - \ref{ex_7} and \ref{ex_8}, the sequence $(b_k)$ may be chosen to grow or decay arbitrarily slowly. Thus, for a radial weight $w$ on $\D$ which is essentially bounded above and below by finite positive numbers on each disc $\{|z| \le R\}$ with $0<R<1$, if
	\[
		\sup_{I\subset\T\,:\,|I|>\epsilon}
		\frac{\esssup_{T_I}w}{\essinf_{T_I}w}
		\xrightarrow[\epsilon\to0]{} \infty,
	\]
	no matter how slowly, then each of the equivalences in Theorem \ref{thm_equiv_2} can fail. Therefore, the assumption in Theorem \ref{thm_equiv_2} regarding the oscillation of $w$ on top halves is the best possible.
\end{remark}

\begin{example}[$w \in AC \cap B_\infty \setminus B_1$]
	\label{ex_9}
	Let $w(z)=f(1-|z|)$, where $f(t)=2^{-n}$ for $2^{-n-1} < t \le 2^{-n}$ and $n\ge0$. Then $w$ is decreasing, $w(0)=1$, and $w \in AC$, but $w \notin B_1$ since $\essinf_\D w = 0$. By Proposition \ref{prop_dec}, $w \in B_\infty$.
\end{example}

\begin{example}[$w \in AC \setminus B_\infty$]
	\label{ex_10}
	Let $w(z)=f(1-|z|)$, where $f(t)=2^n$ for $2^{-n-1} < t \le 2^{-n}$ and $n\ge0$. Then $w$ is increasing, $w(0)=1$, and $w \in AC$, but $w \notin B_\infty$ since $w$ is not integrable:
	\[
		\int_\D w
		= \sum_{n=0}^\infty 2^n ((1-2^{-n-1})^2 - (1-2^{-n})^2)
		\ge \sum_{n=0}^\infty (1-2^{-n})
		= \infty. \qedhere
	\]
\end{example}

The last example in this section shows that the assumption $0<w(0)<\infty$ cannot be removed from Theorem \ref{thm_mono_1}.

\begin{example}
	\label{ex_centre}
	Let $w(z)=|z|^{2r}$, where $r>-1$ and $r\ne 0$. If $r>0$ (resp.\ $r<0$), then $w$ is increasing (resp.\ decreasing) and $w(0)=0$ (resp.\ $w(0)=\infty$). We show that $w \in \bigcup_{1<p<\infty} B_p$ and $w \in RH$ but $w \notin AC$. Note that $w$ is integrable on $\D$ since $r>-1$. Let $f(t) = (1-t)^r$ so that $w(z)=f(1-|z|^2)$. For any $s>-1$, let
	\[
		F_s(x) =
		\begin{cases}
			\frac{1}{x} \int_0^x (1-t)^s\,dt & \text{for~} 0<x\le1, \\
			1 & \text{for~} x=0.
		\end{cases}
	\]
	Then $F_s$ is a continuous function from $[0,1]$ to $(0,\infty)$ (since $\int_0^1 t^s\,dt < \infty$), so there exist constants $0<c_s<C_s<\infty$ such that $c_s \le F_s(x) \le C_s$ for all $0 \le x \le 1$.
	
	Proof of $w \in \bigcup_{1<p<\infty} B_p$: If $r<0$, then $\frac{1}{x} \int_0^x f \le C_r$ and $\essinf_{(0,x]}f \ge 1$ for all $x\in(0,1]$, so $f \in A_1$ and, by Lemma \ref{lem_equiv}, $w \in B_1$. If $r>0$, then, for any $p>r+1$, we have $(1-p')r>-1$ and hence
	\[
		\left(\frac{1}{x} \int_0^x f\right) \left(\frac{1}{x} \int_0^x f^{1-p'}\right)^{p-1}
		\le C_r C_{(1-p')r}^{p-1}
	\]
	for all $x\in(0,1]$, so $f \in A_p$ and hence $w \in B_p$. But note that, for $p=r+1$, we have $(1-p')r=-1$ and hence $\int_0^1 f^{1-p'} = \int_0^1 t^{-1}\,dt = \infty$, so $f \notin A_{r+1}$ and hence $w \notin B_{r+1}$.
	
	Proof of $w \in RH$: Fix $q>1$; if $r<0$, choose $q$ close enough to $1$ that $qr>-1$. Then
	\[
		\frac{(\frac{1}{x} \int_0^x f^q)^{1/q}}{\frac{1}{x} \int_0^x f}
		\le \frac{C_{qr}^{1/q}}{c_r}
	\]
	for all $x\in(0,1]$, so $f \in RH$ and, by Lemma \ref{lem_equiv}, $w \in RH$.
	
	Proof of $w \notin AC$: Since $w(z)$ tends to either $0$ or $\infty$ as $z\to0$, the weight $w$ is not almost constant on the top half $T_I = \{|z|<1/2\}$ corresponding to $I=\T$.
\end{example}

\begin{remark}
	\label{rmk_B_p_B_q}
	In Example \ref{ex_centre}, we have shown that, for every $1<p<\infty$, there exists a weight $w$ on $\D$ such that $w \in B_q$ for all $p<q<\infty$ but $w \notin B_p$. Thus, $B_p \subsetneq \bigcap_{p<q<\infty} B_q$ for all $1<p<\infty$. In particular, for any $1<p<q<\infty$, we have $B_p \subsetneq B_q$.
\end{remark}

\section{Further \texorpdfstring{$A_\infty$}{A\_infinity} conditions}
\label{sec_further}

Theorem \ref{thm_equiv_1} establishes the equivalence of a number of $A_\infty$ conditions for weights on $\D$ which are almost constant on top halves. The full version of this theorem (see \cite{APR1}) contains two more $A_\infty$ conditions:

\begin{enumerate}
	\item[(e)] For every $\beta\in(0,1)$, there exists $\alpha\in(0,1)$ such that, for every $I\subset\T$ and every measurable set $E \subset Q_I$,
	\[
		|E|<\alpha |Q_I| \implies w(E)<\beta w(Q_I).
	\]
	\item[(f)] There exists $C>0$ such that, for every $I\subset\T$,
	\[
		\int_{Q_I} w \log\left(e+\frac{w}{w_{Q_I}}\right) \le C \int_{Q_I} w.
	\]
\end{enumerate}

In this section, we establish the precise relationships between these two conditions and \ref{P1} - \ref{P8} for arbitrary weights on $\D$. First, we show that condition (f) is always equivalent to \ref{P6}.

\begin{lemma}
	\label{lem_P6}
	Let $w$ be a weight on a measure space $(X,\mu)$ equipped with a basis $\BB$. Then $w\in\PP6$ if and only if there exists $C>0$ such that, for every $B\in\BB$,
	\[
		\int_B w \log\left(e+\frac{w}{w_B}\right)\,d\mu \le Cw(B).
	\]
\end{lemma}

\begin{proof}
	$(\impliedby)$ Simply note that, for any $x\ge0$,
	\[
		\log^+(x) \le \max(\log x,1) \le \log(e+x).
	\]
	$(\implies)$ From the concavity of $\log$, it can easily be deduced that, for any $x,y\ge0$,
	\begin{equation}
		\label{eq_log}
		\log(1+x+y) \le \log(1+x) + \log(1+y).
	\end{equation}
	Taking $y=e-1$ in \eqref{eq_log}, we get $\log(e+x) \le 1 + \log(1+x)$ for all $x\ge0$. Hence, for any $B\in\BB$,
	\[
		\int_B w \log\left(e+\frac{w}{w_B}\right)
		\le \int_B w + \int_B w \log\left(1+\frac{w}{w_B}\right).
	\]
	Taking $y=1$ in \eqref{eq_log}, we get $\log(1+x) \le \log 2 + \log x$ for all $x\ge1$. Hence, for any $B\in\BB$,
	\begin{align*}
		\int_B w \log\left(1+\frac{w}{w_B}\right)
		&\le \int_{B\cap\{w\le w_B\}} w \log2 + \int_{B\cap\{w>w_B\}} w \left(\log2+\log\left(\frac{w}{w_B}\right)\right) \\
		&= (\log2) \int_B w + \int_B w \log^+\left(\frac{w}{w_B}\right). \qedhere
	\end{align*}
\end{proof}

Now, we turn our attention to condition (e).

\begin{definition}
	\label{def_P4ab}
	Let $w$ be a weight on a measure space $(X,\mu)$ equipped with a basis $\BB$. We define the following two properties, which are both variations on \ref{P4}:
	\begin{enumerate}
		\item[\namedlabel{P4a}{\bfseries(P4a)}] For every $\beta\in(0,1)$, there exists $\alpha\in(0,1)$ such that, for every $B\in\BB$ and every measurable set $E \subset B$,
		\[
			\mu(E)<\alpha\mu(B) \implies w(E)<\beta w(B).
		\]
		\item[\namedlabel{P4b}{\bfseries(P4b)}] For every $\alpha\in(0,1)$, there exists $\beta\in(0,1)$ such that, for every $B\in\BB$ and every measurable set $E \subset B$,
		\[
			\mu(E)<\alpha\mu(B) \implies w(E)<\beta w(B).
		\]
	\end{enumerate}
\end{definition}

Note that \ref{P4a} is precisely the generalization of condition (e) to arbitrary measure spaces. We are also interested in the similar property \ref{P4b}, which appears in \cite{GR}*{Section IV.2} as a step in the proof of the reverse H\"older inequality for $A_p$ weights on $\R^n$.

It is easy to see that \ref{P4a} lies between \ref{P3} and \ref{P4}, and \ref{P4b} lies between \ref{P1} and \ref{P4}. (Recall that \ref{P1} and \ref{P3} are equivalent to \ref{P1'} and \ref{P3'}, respectively, and see Lemma \ref{lem_P4} below.) However, this observation is not enough to determine precisely where in Figure \ref{fig_map_2} properties \ref{P4a} and \ref{P4b} belong. It turns out that \ref{P4a} lies between \ref{P6} and \ref{P4}, and \ref{P4b} lies between \ref{P2} and \ref{P5} (see Figure \ref{fig_map_P4}), as we now show.

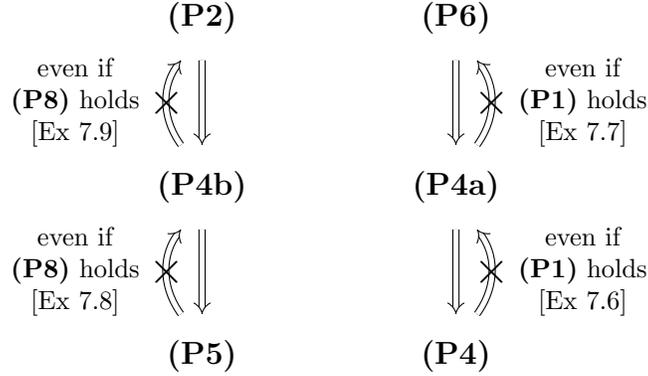
\begin{figure}
	\begin{tikzpicture}[scale=1.125]
		\def \cross {\Large $\times$};
		\tikzset{arrow/.style={-implies, double equal sign distance}};
		\tikzset{txt/.style={font=\Small, align=center}};
		
		\draw (2,8) node {\ref{P2}};
		\draw (2,6) node {\ref{P4b}};
		\draw (2,4) node {\ref{P5}};
		\draw (5,8) node {\ref{P6}};
		\draw (5,6) node {\ref{P4a}};
		\draw (5,4) node {\ref{P4}};
		
		\draw[arrow] (2,7.5) to (2,6.5);
		\draw[arrow] (2,5.5) to (2,4.5);
		\draw[arrow] (5,7.5) to (5,6.5);
		\draw[arrow] (5,5.5) to (5,4.5);
		
		\draw[arrow] (1.75,6.5) to [bend left=36pt] node {\cross}
		node[txt, left=6pt] {even if \\ \ref{P8} holds \\ {[Ex \ref{ex_P4_4}]}} (1.75,7.5);
		\draw[arrow] (1.75,4.5) to [bend left=36pt] node {\cross}
		node[txt, left=6pt] {even if \\ \ref{P8} holds \\ {[Ex \ref{ex_P4_3}]}} (1.75,5.5);
		\draw[arrow] (5.25,6.5) to [bend right=36pt] node {\cross}
		node[txt, right=6pt] {even if \\ \ref{P1} holds \\ {[Ex \ref{ex_P4_2}]}} (5.25,7.5);
		\draw[arrow] (5.25,4.5) to [bend right=36pt] node {\cross}
		node[txt, right=6pt] {even if \\ \ref{P1} holds \\ {[Ex \ref{ex_P4_1}]}} (5.25,5.5);
	\end{tikzpicture}
	\caption{Relationships between \ref{P4a} and \ref{P1} - \ref{P8}, and between \ref{P4b} and \ref{P1} - \ref{P8}, for arbitrary weights on the unit disc.}
	\label{fig_map_P4}
\end{figure}

The proof of the following lemma is based on known ideas (see e.g.\ \cite{G}*{Theorem 7.3.3}); we include it for the convenience of the reader.

\begin{lemma}
	\label{lem_P4}
	Let $w$ be a weight on a measure space $(X,\mu)$ equipped with a basis $\BB$.
	\begin{enumerate}[label=\upshape(\alph*)]
		\item $w\in\PP4$ if and only if there exist $\alpha,\beta\in(0,1)$ such that, for every $B\in\BB$ and every measurable set $E \subset B$,
		\begin{equation}
			\label{eq_P4_alt}
			w(E) < \alpha w(B) \implies \mu(E) < \beta \mu(B).
		\end{equation}
		\item $w\in\PP{4a}$ if and only if, for every $\alpha\in(0,1)$, there exists $\beta\in(0,1)$ such that, for every $B\in\BB$ and every measurable set $E \subset B$, \eqref{eq_P4_alt} holds.
		\item $w\in\PP{4b}$ if and only if, for every $\beta\in(0,1)$, there exists $\alpha\in(0,1)$ such that, for every $B\in\BB$ and every measurable set $E \subset B$, \eqref{eq_P4_alt} holds.
	\end{enumerate}
\end{lemma}

\begin{proof}
	We prove (a) only; the proofs of (b) and (c) are similar. Suppose $w\in\PP4$. Let $\alpha$ and $\beta$ be as in $\PP4$. Given $B\in\BB$ and a measurable set $E \subset B$, let $F = B \setminus E$. Then $F$ is a measurable subset of $B$, so
	\[
		\mu(F) < \alpha \mu(B) \implies w(F) < \beta w(B).
	\]
	Since $\mu(F) = \mu(B) - \mu(E)$ and $w(F) = w(B) - w(E)$, this can be rewritten as
	\[
		\mu(E) > (1-\alpha) \mu(B) \implies w(E) > (1-\beta) w(B).
	\]
	The contrapositive statement is
	\[
		w(E) \le (1-\beta) w(B) \implies \mu(E) \le (1-\alpha) \mu(B).
	\]
	Choose $0<\tilde{\alpha}<1-\beta$ and $1-\alpha<\tilde{\beta}<1$. Then \eqref{eq_P4_alt} holds with $\alpha$ and $\beta$ replaced by $\tilde{\alpha}$ and $\tilde{\beta}$, respectively. The converse can be proved similarly.
\end{proof}

\begin{proposition}
	\label{prop_P4}
	For weights on a measure space equipped with a basis, the following implications hold:
	\begin{enumerate}[label=\upshape(\alph*)]
		\item $\PP6 \implies \PP{4a} \implies \PP4$.
		\item $\PP2 \implies \PP{4b} \implies \PP5$.
	\end{enumerate}
\end{proposition}

\begin{proof}
	(a) The proof of $\PP6 \implies \PP4$ in Theorem 4.1 in \cite{DMO1} can easily be modified to yield a proof of the first implication. The second implication is trivial.
	
	(b) First, suppose $w\in\PP2$. Let $C$ be as in $\PP2$. As in the proof of Theorem 1 in \cite{H}, it can be shown that, for any $B\in\BB$ and any $E \subset B$ with $0<\mu(E)<\mu(B)$, we have
	\[
		\frac{\mu(E)}{\mu(B)} \log\left(\frac{w(B)}{w(E)}-1\right)
		\le \log(2C).
	\]
	Given $\beta\in(0,1)$, choose $\alpha\in(0,1/2)$ small enough that $\beta \log((1/\alpha)-1) > \log(2C)$. Then \eqref{eq_P4_alt} holds for all $B\in\BB$ and all $E \subset B$, so $w\in\PP{4b}$ by Lemma \ref{lem_P4}.
	
	Now, suppose $w\in\PP{4b}$. For $\alpha=3/4$, choose $\beta$ such that the implication in \ref{P4b} holds. Given $B\in\BB$, let $m=m(w;B)$. By the definition of the median, $\mu(B \cap \{w>m\}) < \alpha \mu(B)$, so $w(B \cap \{w>m\}) < \beta w(B)$. Hence $(1-\beta)w(B) < w(B \cap \{w \le m\}) \le m \mu(B)$, so $w_B \le m/(1-\beta)$. Thus, $w\in\PP5$.
\end{proof}

The reader probably wonders if any of the implications in Proposition \ref{prop_P4} can be reversed. It turns out that all four reverse implications are false. Moreover, all four counterexamples can be taken to be weights on $\D$. To prove this, we need the following lemma, whose proof is essentially the same as that of Lemma \ref{lem_equiv} (with $(\star)=\PP4$) and will therefore be omitted.

\begin{lemma}
	\label{lem_equiv_P4}
	Let $w$ and $f$ be as in Lemma \textup{\ref{lem_dict}}, i.e.\ $w(z)=f(1-|z|^2)$. Then, for $(\star)=\PP{4a}$ and $(\star)=\PP{4b}$, we have $w \in (\star) \iff f \in (\star)$.
\end{lemma}

In each of the examples below, for $n\ge0$, we define $E_n = \bigsqcup_{k=n}^\infty (2^{-k-1},2^{-k-1}(1+a_k)]$, so that $E_n \subset (0,2^{-n}]$ with $|E_n| = \sum_{k=n}^\infty 2^{-k-1}a_k$ and $f(E_n) = \sum_{k=n}^\infty 2^{-k-1}a_kb_k$, where $f$ is the function defined in Example \ref{ex_main}.

\begin{example}[$w\in\PP1\cap\PP4\setminus\PP{4a}$]
	\label{ex_P4_1}
	Let $w$ be as in Example \ref{ex_1}. Then $w\in\PP1$ and hence $w\in\PP4$.
	
	Proof of $w\notin\PP{4a}$: Note that $|E_n| \le 2^{-n}a_n$ (since $a_k \searrow 0$) and $f(E_n) = 2^{-n}$ (since $a_kb_k=1$). Also, $f((0,2^{-n}] \setminus E_n) = |(0,2^{-n}] \setminus E_n| \le 2^{-n}$, so $f((0,2^{-n}]) \le 2^{1-n}$. Let $\beta=1/4$. Given $\alpha\in(0,1)$, choose $n$ such that $a_n<\alpha$. Then $|E_n|<\alpha|(0,2^{-n}]|$, but $f(E_n)>\beta f((0,2^{-n}])$, so $f\notin\PP{4a}$. By Lemma \ref{lem_equiv_P4}, $w\notin\PP{4a}$.
\end{example}

\begin{example}[$w\in\PP1\cap\PP{4a}\setminus\PP6$]
	\label{ex_P4_2}
	In Example \ref{ex_main}, let $b_0 \ge e$, $b_k \nearrow \infty$, and $a_k = 1/(b_k(1+\log\log{b_k}))$. As in Example \ref{ex_1}, $w\in\PP1$ and $w\notin\PP6$.
	
	Proof of $w\in\PP{4a}$: Note that $|E_n| \ge 2^{-n-1}a_n$ and $f(E_n) \le 2^{-n}a_nb_n$ (since $a_kb_k\searrow0$). Given $\beta\in(0,1)$, first choose $N$ such that $a_{N+1}b_{N+1}<\beta/4$ and then choose $0<\alpha<\min(2^{-N-2}a_{N+1},\beta/4)$. Let $x\in(0,1]$ and $E \subset (0,x]$ be given. Then $2^{-n-1}<x\le 2^{-n}$ for some $n\ge0$. Since $f\ge1$, we have $f((0,x]) \ge x$. Suppose $|E|<\alpha |(0,x]|$.
	
	Case 1: Suppose $n \le N$. Then $|E| \le 2^{-N-2}a_{N+1} \le |E_{N+1}|$. Since $f(s) \ge f(t)$ for $s \in E_{N+1}$ and $t \notin E_{N+1}$, this implies that
	\[
		f(E) \le f(E_{N+1})
		\le 2^{-n-1}a_{N+1}b_{N+1}
		< \beta x
		\le \beta f((0,x]).
	\]
	
	Case 2: Suppose $n \ge N+1$. Since $f\ge1$ on $E_n$ and $f=1$ on $(0,2^{-n}] \setminus E_n$, and since $|E| \le \alpha 2^{-n}$, we have
	\[
		f(E) \le f(E_n) + \alpha 2^{-n}
		\le 2^{-n}a_{N+1}b_{N+1} + \alpha 2^{-n}
		< 2^{-n-1} \beta
		\le \beta f((0,x]).
	\]
	
	In either case, $f(E) < \beta f((0,x])$. Thus, $f\in\PP{4a}$. By Lemma \ref{lem_equiv_P4}, $w\in\PP{4a}$.
\end{example}

\begin{example}[$w\in\PP8\cap\PP5\setminus\PP{4b}$]
	\label{ex_P4_3}
	In Example \ref{ex_main}, let $b_0\le1$, $b_k\searrow0$, and $a_k=1/4$. As in Example \ref{ex_3}, $w\in\PP8$ and $w\in\PP5$.
	
	Proof of $w\notin\PP{4b}$: Note that $|E_n| = 2^{-n-2}$ and $f(E_n) \le 2^{-n-2}b_n$. Also, $f((0,2^{-n}]) \ge f((0,2^{-n}] \setminus E_n) = |(0,2^{-n}] \setminus E_n| = 3\cdot2^{-n-2}$. Let $\alpha=7/8$. Given $\beta\in(0,1)$, choose $n$ such that $b_n<3(1-\beta)$. Then $|E_n| > (1-\alpha)|(0,2^{-n}]|$, but $f(E_n) < (1-\beta) f((0,2^{-n}])$. Let $F_n = (0,2^{-n}] \setminus E_n$. Then $|F_n| < \alpha|(0,2^{-n}]|$, but $f(F_n) > \beta f((0,2^{-n}])$. Thus, $f\notin\PP{4b}$. By Lemma \ref{lem_equiv_P4}, $w\notin\PP{4b}$.
\end{example}

\begin{example}[$w\in\PP8\cap\PP{4b}\setminus\PP2$]
	\label{ex_P4_4}
	Let $w$ be as in Example \ref{ex_3}. Then $w\in\PP8$ and $w\notin\PP2$.
	
	Proof of $w\in\PP{4b}$: Note that $|E_n| \le 2^{-n}a_n$ (since $a_k\searrow0$) and $f(E_n) \ge 2^{-n-1}a_nb_n$. Given $\beta\in(0,1)$, choose $N$ such that $a_N < \beta/4$ and choose $0<\alpha<2^{-N-1}a_Nb_N$. Let $x\in(0,1]$ and $E \subset (0,x]$ be given. Then $2^{-n-1} < x \le 2^{-n}$ for some $n\ge0$. Since $f\le1$, we have $f((0,x]) \le x$. Suppose $|E| \ge \beta |(0,x]|$.
	
	Case 1: Suppose $n \le N$. Then $|E| \ge 2^{-n-1} \beta \ge 2^{-N}a_N \ge |E_N|$. Since $f(s) \le f(t)$ for $s \in E_N$ and $t \notin E_N$, this implies that
	\[
		f(E) \ge f(E_N)
		\ge 2^{-N-1}a_Nb_N
		\ge \alpha
		\ge \alpha f((0,x]).
	\]
	
	Case 2: Suppose $n \ge N+1$. Then $|E| \ge 2^{1-n}a_N \ge |E_n| + 2^{-n}a_N$. Since $f\le1$ on $E_n$ and $f=1$ on $(0,2^{-n}] \setminus E_n$, we have
	\[
		f(E) \ge f(E_n) + 2^{-n}a_N
		\ge 2^{-n}a_N
		\ge 2^{-n}\alpha
		\ge \alpha f((0,x]).
	\]
	
	In either case, $f(E) \ge \alpha f((0,x])$. Thus, $f(E)<\alpha f((0,x]) \implies |E|<\beta|(0,x]|$. By Lemma \ref{lem_P4}, $f\in\PP{4b}$. By Lemma \ref{lem_equiv_P4}, $w\in\PP{4b}$.
\end{example}

Together, Theorem \ref{thm_equiv_2} and Proposition \ref{prop_P4} imply that \ref{P1} - \ref{P8}, \ref{P4a} and \ref{P4b} are equivalent for weights on $\D$ which are almost constant on top halves. Comparing this result with Examples \ref{ex_P4_1} - \ref{ex_P4_4}, we see once again that weights which are almost constant on top halves are much more well-behaved than weights which are not. Furthermore, in each of these examples, the sequence $(b_k)$ can be taken to tend to $0$ or $\infty$ arbitrarily slowly, so being almost constant on top halves is indeed the sharp condition for this good behaviour.

\section{Acknowledgements}

I would like to sincerely thank my supervisor Prof.\ Ignacio Uriarte-Tuero for suggesting the topic and for his support and guidance throughout this research project. I am also grateful to Prof.\ Alberto Dayan, Prof.\ Adri\'an Llinares, and Prof.\ Karl-Mikael Perfekt for permitting me to use an adapted version of their diagram of a Carleson square. In addition, I thank the anonymous referee for his/her comments and suggestions.

This research was partially supported by an NSERC Undergraduate Student Research Award.

\end{document}